\documentclass[11pt,reqno]{amsart}
\usepackage{stmaryrd}
\usepackage[margin=1in]{geometry}
\usepackage[colorlinks,linkcolor=magenta,citecolor=blue]{hyperref}
\usepackage{amssymb}
\usepackage{amsfonts}
\usepackage{amsmath}
\usepackage{amsthm}
\usepackage{color}
\usepackage{graphicx}
\usepackage{epsfig,mathrsfs}
\usepackage[bf,SL,BF]{subfigure}
\usepackage{fancyhdr}
\usepackage{graphicx}
\usepackage{subfigure}
\usepackage{caption}
\usepackage{wrapfig}
\usepackage{cases}
\usepackage{makecell}
\usepackage{multirow}
\usepackage{algorithm}
\usepackage{enumitem}
\usepackage{algorithmic}
\usepackage{comment}
\allowdisplaybreaks

\newtheorem{theorem}{Theorem}
\newtheorem{lemma}[theorem]{Lemma}

\newtheorem{remark}{Remark}

\numberwithin{equation}{section}

\numberwithin{theorem}{section}

\numberwithin{subsection}{section}

\newcommand{\R}{\mathbb{R}}
\newcommand{\N}{\mathbb{N}}
\newcommand{\pa}{\partial}
\newcommand{\LOM}[1]{L^{#1}(M)} 
\newcommand{\LQM}[1]{L^{#1}(M_T)} 
\newcommand{\LO}[1]{L^{#1}(\Omega)}
\newcommand{\LM}[1]{L^{#1}(M)}
\newcommand{\LQ}[1]{L^{#1}(Q_T)}

\newcommand{\sumi}{\sum_{i=1}^{m_1}}
\newcommand{\sumj}{\sum_{j=1}^{m_2}}
\newcommand{\intO}{\int_{\Omega}} 
\newcommand{\intOM}{\int_{M}}
\newcommand{\intQ}{\int^T_0\int_{\Omega}}

\newcommand{\intQM}{\int^T_0\int_{M}}

\newcommand{\eps}{\varepsilon}

\newcommand{\M}{M} 
\newcommand{\sumii}{\sum_{i=1}^{m_1}}
\newcommand{\sumjj}{\sum_{j=1}^{m_2}}

\newcommand{\norm}[1]{\left\|#1\right\|}

\def\({\left(}
\def\){\right)}

\newcommand{\rM}{r_{M}}
\newcommand{\rO}{r_{\Omega}}

\newcommand{\mM}{\mu_{M}}

\newcommand{\intQtautwo}{\int^{\tau+2}_{\tau}\int_{\Omega}}

\newcommand{\intMtautwo}{\int^{\tau+2}_{\tau}\int_{M}}
\newcommand{\vat}{\varphi_\tau}

\renewcommand{\H}{\mathscr{H}}

\newcommand{\LS}[1]{L^{#1}(M_T)}
\newcommand{\LStau}[1]{L^{#1}(M_{\tau,\tau+1})}
\newcommand{\LStaut}[1]{L^{#1}(M_{\tau,\tau+2})}
\newcommand{\LQtau}[1]{L^{#1}(Q_{\tau,\tau+1})}
\newcommand{\LQtaut}[1]{L^{#1}(Q_{\tau,\tau+2})}

\title[Volume-surface systems]{Volume-surface systems with sub-quadratic intermediate sum on the surface: Global existence and boundedness}
\author[J. Yang]{Juan Yang}
\address{Juan Yang \hfill\break
School of Mathematics and Statistics, 
Lanzhou University, Lanzhou, 730000, PR China\hfill\break
Department of Mathematics and Scientific Computing, University of Graz, \\  Heinrichstrasse 36, 8010 Graz, Austria}
\email{yangjuan18@lzu.edu.cn, yangjuan0912@gmail.com}
\author[B.Q. Tang]{Bao Quoc Tang}
\address{Bao Quoc Tang \hfill\break
Institute of Mathematics and Scientific Computing, University of Graz, 
Heinrichstrasse 36, 8010 Graz, Austria}
\email[Corresponding author]{quoc.tang@uni-graz.at, baotangquoc@gmail.com}
\begin{document}
\begin{abstract}
	The global existence and boundedness of solutions to volume-surface reaction diffusion systems with a mass control condition are investigated. Such systems arise typically in e.g. cell biology, ecology or fluid mechanics, when some concentrations or densities are inside a domain and some others are on its boundary. Comparing to previous works, the difficulty of systems under consideration here is that the nonlinearities on the surface can have a sub-quadratic growth rates in all dimensions. To overcome this, we first use the Moser iteration to get some uniform bounds of the time integration of the solutions. Then by combining these bounds with an $L^p$-energy method and a duality argument, we obtain the global existence of solutions. Moreover, under mass dissipation conditions, the solution is shown to be bounded uniformly in time. 
\end{abstract}

\maketitle
\noindent{\small{\textbf{Classification AMS 2010:} 35A01, 35K57, 35K58, 35Q92.}}

\noindent{\small{\textbf{Keywords:} Volume-surface reaction-diffusion systems; Global existence; Moser iteration; Duality method; Intermediate sum condition;}}


\section{Introduction and main results} 

Reaction-diffusion systems play a crucial role in modeling various phenomena in biological, chemical, and physical processes, where interactions between reactive substances and their diffusion in space are of particular interest. A special class of such systems called volume-surface (or bulk-surface) systems, arises naturally in contexts where concentrations or densities occur not only within a domain but also on its boundary. In recent years, important progress has been made in studying reaction-diffusion systems, and the well-posedness of volume-surface systems is also an active area of research. These systems are more complex because of the coupling between volume and surface dynamics, so we should carefully deal with the boundary conditions and how diffusion and reaction terms work together. In this paper, the global existence of mass controlled volume-surface systems is investigated, where the nonlinearities on the surface can have sub-quadratic growth rates in all dimensions.

\subsection{Problem setting and state of the art} We start with some description of the system under consideration and review existing literature concerning its global well-posedness. Let $m_1, m_2\in \mathbb N_+$. We consider the following volume-surface reaction-diffusion system for $u = (u_1, \ldots, u_{m_1})$, and $v = (v_1, \ldots, v_{m_2})$,
\begin{equation}\label{System}
	\left\{
	\begin{aligned}
		\pa_t u_{i} &=d_{i}\Delta u_{i}+F_{i}(u), &&\text{ on }  \Omega_T:=\Omega\times(0,T), i=1,\ldots, m_1,\\
		d_{i}\pa_{\nu} u_{i} &=G_{i}(u,v), &&\text{ on } \M_T:= \M\times(0,T), i=1,\ldots, m_1,\\
		\pa_t v_{j} &=\delta_{j}\Delta_{\M} v_{j}+H_{j}(u,v), && \text{ on }  \M_T:= \M\times(0,T),j=1,\ldots, m_2,\\
		u_i(0)&=u_{i,0} && \text{ on  } \overline{\Omega}\times{\{0\}}, i=1,\ldots, m_1,\\
		v_j(0)&=v_{j,0} && \text{ on  } \M\times{\{0\}}, j=1,\ldots, m_2, 
	\end{aligned}\right. 
\end{equation} 
where  $\Omega\subset\R^n$, $n\geq 2$, be a bounded domain with smooth (e.g. $C^{2+\delta}$ for some $\delta>0$) boundary $\M:= \partial\Omega$,   $\pa_{\nu}$ denotes the outward normal flux on the boundary $\M$, the diffusion coefficients $d_i,~\delta_j$ are positive for all  $i=1,\cdots, m_1$, $j=1,\cdots, m_2$. The nonlinearities are assumed to satisfy the following assumptions.
\begin{enumerate}[label=(A\theenumi),ref=A\theenumi]
	\item\label{A1} (Local Lipschitz continuity) $F_i: \R_{+}^{m_1} \to \R$, $G_i: \R_+^{m_1+m_2}\to \R$, $H_j: \R_+^{m_1+m_2}\to \R$ are locally Lipschitz continuous for $i\in \{1,\cdots , m_1\}$ and $j\in \{1, \cdots , m_2\}$;
	\item\label{A2} (Quasi-positivity) 
	\begin{align*}
		F_{i}(u) \ge0 &\text{ for all }u\in \R_{+}^{m_1} \text{ with }u_{i}=0, \; \forall i=1,\cdots , m_1,\\
		G_{i}(u,v)\ge0  &\text{ for all }u\in \R_{+}^{m_1}, v\in \R_{+}^{m_2}\text{ with }u_{i}=0, \; \forall i=1,\cdots , m_1, \\
		H_j(u,v) \geq 0  &\text{ for all } u\in \R_+^{m_1}, v\in \R_+^{m_2} \text{ with } v_{j} = 0, \; \forall j=1,\cdots , m_2;
	\end{align*}
	The ubiquitous assumption \eqref{A1} ensures the local existence of a solution, while \eqref{A2} presents some realistic physical property: if some concentration is zero, it cannot be consumed in a reaction. An important consequence of \eqref{A2} is that the solution will preserve the non-negativity of initial data.
	\item \label{A3} (Mass control) There exist $L\in\R$, $K\geq 0$, positive constants $\alpha_1, \ldots, \alpha_{m_1}, \beta_1, \ldots, \beta_{m_2}$ such that
	\begin{equation}
		\begin{aligned}
			\sumii   \alpha_i F_i(u)&\le L\left(|u| + 1 \right)  ,\\ 
			\sumii   \alpha_i G_i(u,v) &\le L\left(|u| + |v| + 1 \right),\\
			\sumjj   \beta_j H_j(u,v)&\le L\left(|u| + |v| + 1 \right) 
		\end{aligned}
	\end{equation}
	where $|u| = \sum_{i=1}^{m_1}u_i$ and $|v| = \sum_{j=1}^{m_2}v_j$,
	for all $u\in \R_+^{m_1}$ and $v\in \R_+^{m_2}$. Assumption \eqref{A3} appears frequently in applications, as it shows certain ``balancing'' property, in the sense that, if some nonlinearities appear as a source (i.e. with the plus sign) in some equation, it should be a product in another equation (i.e. with the minus sign). A mathematical consequence of this assumption is that the total mass, i.e. $L^1$-norm, is controlled for all time. Without loss of generality, \textit{we can assume for the rest of this paper} that $\alpha_i = 1$ and $\beta_j = 1$ for all $i=1,\ldots, m_1$ and all $j=1,\ldots, m_2$.
	
	\item\label{A4} (Intermediate sum condition)  There exist a lower triangular matrix $A\in \R^{(m_1+m_2)\times(m_1+m_2)}$ with non-negative elements and positive diagonal elements, and constants  $K_1\geq 0$, $\rO>0$, $\rM>0$, $\mu_{\M}>0$, 
	such that the following component-wise inequalities are fulfilled 
	\begin{align*} 
		\begin{aligned}
			A
			\begin{bmatrix}
				F(u)\\ \vec{0}_{m_2}
			\end{bmatrix}
			&\le K_1 \left(\sumii u_i^{\rO}+1 \right)\vec{1}_{m_1+m_2},\\
			A
			\begin{bmatrix}
				G(u,v)\\H(u,v)
			\end{bmatrix}
			&\le K_1 \begin{bmatrix}
				\left(|u|^{r_M} + |v|^{r_M}+1\right)\vec{1}_{m_1}\\
				\left(|u|^{\mu_M} + |v|^{\mu_M}+1\right)\vec{1}_{m_2}
			\end{bmatrix}
		\end{aligned}
	\end{align*}
	for all $u\in \R_{+}^{m_1}$ and $v\in \R_{+}^{m_2}$,
	where  $\vec{1}_{m}$ (res. $\vec{0}_{m}$) is the column vector of all $1s$ elements (res. $0s$ elements) in $\R^m$, and the exponents $\rO$, $\rM$ and $\mu_M$ satisfy some bounds (see Theorem \ref{Theorem:1}). This is the only technical assumption of the nonlinearities, though they still appear frequently in applications (see e.g. \cite{morgan2023global,augner2024analysis} for more details). It is noted that only the first nonlinearities, i.e. $F_1(u)$ and rsp. $G_1(u,v)$ need to be bounded above by polynomials of order $\rO$ and rsp. $r_M$, while all other nonlinearities just need to satisfy some good ``cancellation'' through linear combinations by elements of the matrix $A$.
	
	\item\label{A5} (Polynomial bound) There are constants $K_2, r> 0$ such that, for all $i=1,\cdots , m_1$, $j=1,\cdots , m_2$,
	\begin{equation*} 
		F_i(u)\leq K_2\left(|u|^{r}+ 1\right), \quad G_i(u,v),~H_j(u,v)\leq K_2\left(|u|^{r}+|v|^{r}+1\right) 
	\end{equation*}
	for all $u\in \R_{+}^{m_1}$ and $v\in \R_{+}^{m_2}$. It is remarked that \textit{there is no restriction on $r$}, that means the nonlinearities generally can have arbitrary polynomial growth rates provided \eqref{A4} is fulfilled.
\end{enumerate}

\medskip
Volume-surface systems of type \eqref{System} have recently attracted a lot of attention, not only due to  their numerous applications in biology \cite{alphonse2018coupled,fellner2016quasi,borgqvist2021cell,garcia2014mathematical}, ecology \cite{berestycki2013influence,berestycki2015effect}, heterogeneous catalysts \cite{gesse2024stability}, or fluid mechanics \cite{mielke2012thermomechanical} and many others, but also because of the highly non-trivial challenges in the analysis of such systems. Thanks to the local Lipschitz continuity of nonlinearities in \eqref{A1}, the unique local strong solution on a maximal interval $(0,T_{\max})$ is quite standard, and one has the blow-up criterion 
\begin{align}\label{Blow_up}
	T_{\max}<+\infty \quad \Longrightarrow\quad \limsup_{t\rightarrow T_{\max}} \Big(\sumii\|u_i(t)\|_{\LO{\infty}} + \sumjj\|v_j(t)\|_{\LOM{\infty}} \Big) = \infty.
\end{align}
Thanks to  the quasi-positivity condition \eqref{A2}, the solution of \eqref{System} is non-negative if the initial data is also non-negative. 
Moreover, by using \eqref{A3} and a duality argument, one can obtain the estimate
\begin{align}\label{0bound}
	\sumii \|u_i\|_{L^{\infty}(0,T;\LO{1})} +  \sumjj \|v_j\|_{L^{\infty}(0,T;\LOM{1})} + \sumii \|u_i\|_{\LQM{1}} \leq C_T.
\end{align}
Extending this local solutio globally, on the other hand, a significantly challenging issue. Indeed, comparing with \eqref{Blow_up}, the bounds \eqref{0bound} are far from sufficient to show the global existence of the system \eqref{System}. In fact, even in the case of classical reaction-diffusion systems, i.e. $m_2 = 0$, the authors of \cite{pierre2000blowup} already showed that the solution can possibly blow up in finite time under \textit{only} the natural assumptions \eqref{A1}, \eqref{A2} and \eqref{A3}. Therefore, one need to impose more conditions on the nonlinearities, see the extensive survey \cite{Pierre2010Global} for more details. The study of volume-surface systems is even more challenging, due to the volume-surface coupling, which makes some methods for classical reaction-diffusion systems fail to apply, see e.g. \cite{morgan2019martin}.
Let us review some existing results concerning the global existence of \eqref{System}.
When the system is linear or the reactions are restricted to at most linear growth, global solutions were proven in \cite{fellner2016quasi,hausberg2018well}. 
In the case of nonlinear coupling on the boundary, global bounded solutions for the reversible reaction $\alpha \mathcal U \leftrightharpoons \beta \mathcal V$ 
are established in \cite{fellner2018well}. 
In the works \cite{sharma2016global,ma2017uniform},  global existence and boundedness of solutions are shown for volume-surface reaction diffusion systems under the assumption of a linear upper bound for $F_i$, $G_i$, and the sum of $G_i + H_j$, i.e. for any $i\in \{1,\cdots , m_1\}$ and any $j\in \{1,\cdots , m_2\}$, 
\begin{equation}\label{linear_growth}
	F_i(u) \leq L (1+|u|) , \quad G_i(u,v) + H_j(u,v) \leq L( 1+|u|+|v|)
\end{equation}
for some $L>0$. If it is known a-priori that control of $L^\infty$-bounds is preserved, then one can get global classical solutions for general systems, see \cite{disser2020global}. It should also be mentioned that there are also many other works on volume-surface systems focusing on stability analysis, pattern formation or singular limit have also been carried out, see e.g. \cite{madzvamuse2015stability,alphonse2018coupled,paquin2019pattern,niethammer2020bulk,gomez2021pattern,ratz2014symmetry,backer2022analysis}. The global existence of systems in these works are usually not the scope and/or can be obtained using the aforementioned works on well-posedness.
Recently, in the work \cite{morgan2023global}, by using some $L^p$-energy functions and duality method, the global existence of solutions for the system \eqref{System} under the assumptions \eqref{A1}--\eqref{A5} with $1\leq\rO<1+ \frac{2}{n}$, $1\leq\rM<1+\frac{1}{n}$ and $1\leq \mu_{\M}\leq 1+ \frac{4}{n+1}$ was obtained. Moreover, if the mass control condition are replaced by the  mass dissipation condition, the solution isuniformly bounded in time.  Similar results are obtained in  \cite{augner2024analysis}, where it  assumed more restrictive conditions on nonlinearties $\rO < 1 + \frac 2n$, $\mu_M < 1 + \frac{2}{n+1}$, and $- K_1(1+|u|)\leq G_i(u,v)\leq K_2(1+|v|)$, but less restrictive initial data. In both of this work, $\mu_M \to 1$ as $n\to \infty$, i.e. the intermediate sum condition for the surface reactions become almost linear for large dimensions $n$. In this paper, we improve these existing results by allowing sum \textit{dimensionally independent bounds} on $\mu_M$, namely $1\leq \mu_{\M}< 2$ arbitrary, which is a clear improvement when $n\geq 4$. In order to obtain this improvement, we combine the techniques from \cite{morgan2023global}, some Moser iteration and a duality argument from \cite{morgan1989global}, which will be detailed in the next section.


\subsection{Main result and key ideas}
The main result of this paper is the following theorem.
\begin{theorem}\label{Theorem:1}
	Let $n\geq 4$. Assume \eqref{A1}, \eqref{A2}, \eqref{A3}, \eqref{A4} and \eqref{A5} with
	\begin{align}\label{Condition:Growth}
		1\leq \rO<1+\frac{2}{n}, \quad 1\leq r_{\M}<1+ \frac{1}{n} \quad\text{and} \quad 1\leq \mu_{\M}<2.
	\end{align} 
	Then, for any nonnegative, bounded initial data $(u_0, v_0)\in (W^{2,p}(\Omega))^{m_1} \times (W^{2,p}(\M))^{m_2}$ for some $p>n$ satisfying the compatibility condition
	\begin{align}\label{compatibility}
		d_i \nabla u_{i0} \cdot \nu = G_i(u_0, v_0) \quad \text{on } \M \quad \text{for all } i=1,\cdots, m_1,
	\end{align}
	the system \eqref{System} has a unique global classical solution. Moreover, if $L<0$ or $L=K=0$ in \eqref{A3}, then
	\begin{align*}
		\sup_{i=1,\cdots,m_1} \sup_{j=1,\cdots,m_2} \sup_{t\geq 0} \left( \|u_i(t)\|_{\LO{\infty}} + \|v_j(t)\|_{\LOM{\infty}} \right) < +\infty.
	\end{align*}
\end{theorem}

\begin{remark}\hfill
	\begin{itemize} 
		\item In \cite{morgan2023global}, for $1\leq \mu_{\M}\leq 1+ \frac{4}{n+1}$ or in \cite{augner2024analysis} for $1\le \mu_M \le 1 + \frac{2}{n+1}$, the global existence of solutions has been proved. It is noted that
		$ 1+ \frac{4}{n+1} < 2$ for higher dimension.  Our results improve on the condition of $\mu_{\M}$ for $n\geq 4$ and the condition does not depend on the dimension $n$. 
		\item Since $1+\frac{2}{n} < 1+ \frac{4}{n+1}$ for $n\geq 4$, the result has been obtained in \cite{morgan2023global} if $1\leq \mu_{\M} \leq \rO<1+\frac{2}{n}$. Therefore, we may focus on $\rO \leq \mu_{\M} <2$. 
		\item The assumption of $W^{2,p}$-initial data for $p>n$ is to obtain local existence of solutions, which was shown in e.g. \cite{sharma2016global}. Reducing the bounded initial data or $W^{1,p}$-initial for smaller $p$, see e.g. \cite{augner2024analysis}, is left for future investigation.
	\end{itemize}
\end{remark}

Let us describe the ideas to prove Theorem \ref{Theorem:1}.  
Firstly, the condition \eqref{A3} gives
\begin{align}\label{Intro:L1}
	\|u_i\|_{L^\infty(0,T;\LO{1})} + \|u_i\|_{L^1(M\times(0,T))} + \|v_j\|_{L^\infty(0,T;\LM{1})} \leq C_T.
\end{align}
In order to deal with $\mu_{\M}<2$ arbitrary, we extend the duality argument in \cite{morgan1989global} to the case of volume-surface systems. This duality method requires some bounds for the time integration of the solution, more precisely
\begin{align}\label{Intro:L_xL_t}
	\Big\|\int^t_0 u_i \,ds\Big\|_{\LO{\infty}} + \Big\|\int^t_0 u_i \,ds\Big\|_{\LOM{\infty}} + \Big\|\int^t_0 v_j \,ds\Big\|_{\LOM{\infty}} \leq C_T
\end{align} 
where $C_T$ is a constant depending on $T>0$.
For classical mass dissipating reaction-diffusion systems in a domain, i.e. $m_2 = 0$, one has (assuming mass dissipation)
\begin{align*}
	\pa_t \Big(\sumi u_i\Big) - \Delta \Big(\sumi d_i u_i\Big) \leq 0.
\end{align*}
Integrating over $(0,t)$, we have
\begin{align*}
	\sumi  u_i(t)  - \Delta \Big(\int^t_0 \sumi d_i u_i \,ds\Big) \leq \sumi u_{i0},
\end{align*}
with homogeneous Neumann boundary conditions. Thus, an $\LO{\infty}$-bound for $\int^t_0 u_i \,ds$ can be easily obtained from comparison principle.    
In our case, similar computations give 
\begin{equation*}
	\left\{
	\begin{aligned}
		& - \Delta\left(\int_0^t\sumi d_iu_i ds\right) \le L\int_0^t \sumi u_i ds + Lt + \sumi u_{i0}, &&x\in\Omega,\\
		&\partial_\nu\left(\int_0^t \sumi d_iu_ids \right) \le L\int_0^t\sumi u_i ds + L\int_0^t\sumj v_jds + Lt, &&x\in M,\\
		&- \Delta_{M}\left(\int_0^t\sumj \delta_j v_jds \right) \le L\int_0^t\sumi u_i ds + L\int_0^t\sumj v_jds + Lt + \sumj v_{j0}, &&x\in M.
	\end{aligned}
	\right.
\end{equation*}
Due to the volume-surface coupling, it seems that the comparison principle is not applicable. Even if it would be, we only get
\begin{align*}
	\Big\|\int^t_0 u_i \,ds\Big\|_{\LO{\infty}}  + \Big\|\int^t_0 v_j \,ds\Big\|_{\LOM{\infty}} \leq C_T.
\end{align*} 
In this paper, we resolve this issue by using an energy method and a Moser iteration, in the spirit of \cite{alikakos1979lp}, to get the desired estimate \eqref{Intro:L_xL_t}. 

\medskip
Let us briefly sketch the main steps of the proof of Theorem \ref{Theorem:1}, which can be roughly divided in several steps:
\begin{itemize}
	\item Step 1:  Firstly, the mass control condition \eqref{A3} gives, by using a duality method adapted to volume-surface systems,
	\begin{align*}
		\|u_i\|_{L^{\infty}(0,T;\LO{1})} + \|u_i\|_{\LQM{1}} + \|v_j\|_{L^{\infty}(0,T;\LOM{1})}\leq C_T.
	\end{align*}
	Then, the intermediate sum condition \eqref{A4} allows us to construct some $L^p$ energy functions and obtain for any $\eps>0$ the following inequality
	\begin{align} \label{Step2}
		\sumii  \left(\|u_i\|_{\LQ{p}}+\|u_i\|_{\LS{p}}\right)\leq C_{p,\eps,T} + \eps\sumjj \|v_j\|_{\LS{p}}
	\end{align} 
	for any positive integer $p\in \mathbb Z_+$.
	\item Step 2: Based on condition \eqref{A3}, we employ the Moser iteration technique to obtain  the important bounds 
	\begin{align*}
		\Big\|\int^t_0 u_i \,ds\Big\|_{\LQ{\infty}} + \Big\|\int^t_0 u_i \,ds\Big\|_{\LQM{\infty}} + \Big\|\int^t_0 v_j \,ds\Big\|_{\LQM{\infty}} \leq C_T.
	\end{align*}

	\item Step 3: By using crucial estimates in Step 2 and adapting the duality argument in \cite{morgan1990boundedness}, there exists for any $\eta>0$ a constant $C_{T,\eta}>0$ such that
	\begin{equation*} 
		\|v_j\|_{\LQM{p}} 
		\leq  C_{T,\eta}   +  \eta \sumjj\|v_j\|_{\LQM{p}} + C \sum_{k=1}^{j-1}  \|v_k\|_{\LQM{p}}.  
	\end{equation*}
	This ultimately leads to the boundedness of $v_j$ in $L^p(M\times(0,T))$, and consequently, thanks to \eqref{Step2}, 
	\begin{align*}
		\|u_i\|_{\LQ{p}} + \|u_i\|_{\LQM{p}} + \|v_j\|_{\LQM{p}} \leq C_T.
	\end{align*}
	Applying the polynomial growth of the nonlinearities in \eqref{A5} and regularization of the heat operator with inhomogeneous boundary conditions, we obtain finally 
	\begin{align*}
		\|u_i\|_{\LQ{\infty}} +  \|v_j\|_{\LQM{\infty}} \leq C_T,
	\end{align*}
	which concludes the global existence of bounded solutions.
	
	\item Step 4:  To obtain uniform-in-time bounds of the global solution, we use smooth cut-off function to obtain  
	\begin{align*}
		\|u_i\|_{\LQtaut{\infty}} +  \|v_j\|_{\LStaut{\infty}} \leq C,
	\end{align*}
	which concludes the uniform-in-time bounds of the global solution, and consequently completes the proof of Theorem \ref{Theorem:1}. 
\end{itemize}

\medskip
\noindent{\bf Organization of the paper.} In the next section, we give the local existence theorem and some useful lemmas. In Section \ref{Sec:GlobalExsitence}, we show that global existence of solutions  for the system \eqref{System}, while the uniform-in-time boundedness of solutions is shown in Section \ref{Sec:uniform-in-time}.

\medskip
\noindent{\bf Notation}. For the rest of this paper, we will use the following notation:
\begin{itemize} 
	\item For $0\leq \tau < T$,
	\begin{equation*}
		Q_{\tau,T}:= \Omega\times (\tau, T) \quad \text{and}\quad M_{\tau,T}:= M\times(\tau,T).
	\end{equation*}
	When $\tau = 0$, we simply write $Q_T$ and $M_T$.
	\item For $1\leq p < \infty$,
	\begin{equation*}
		\|f\|_{L^p(Q_{\tau,T})}:= \left( \int_\tau^T\intO |f(x,t)|^p \right)^{\frac 1p}
	\end{equation*}
	and for $p = \infty$,
	\begin{equation*}
		\|f\|_{L^\infty(Q_{\tau,T})}:= \mathrm{ess\,sup}_{Q_T}|f(x,t)|.
	\end{equation*}
	The spaces $L^p(M_{\tau,T})$ with $1\leq p\leq \infty$ are defined in a similar way.
	\item For $1\leq p \leq \infty$,
	\begin{equation*}
		W^{2,1}_p(Q_{\tau,T}):= \left\{f\in L^p(Q_{\tau,T}): \pa_t^r\pa_x^sf\in L^p(Q_{\tau,T}) \; \forall r,s\in \mathbb N, \, 2r+s\leq 2\right\}
	\end{equation*}
	with the norm
	\begin{equation*}
		\|f\|_{W^{2,1}_p(Q_{\tau,T})}:= \sum_{2r+s\leq 2}\|\pa_t^r\pa_x^s f\|_{L^p(Q_{\tau,T})}.
	\end{equation*}
	\item 
	The constant $C$ represents a generic constant whose value may vary from line to line, or even within the same line, unless explicitly stated. Sometimes it is written e.g. $C_T$ to note the dependence of $C$ on $T$.  
\end{itemize}

\section{Preliminaries} 
In this section, we start with the local existence of solution for the system \eqref{System}, and provide the blow-up criterion, which were proved in \cite{sharma2016global}.  
\begin{theorem}[\cite{sharma2016global}, Theorem 3.2]\label{Theorem:LocalExistence}
	Assume  \eqref{A1}.
	For any smooth initial data $(u_0,v_0)\in (W^{2,p}(\Omega))^{m_1}\times (W^{2,p}(\M))^{m_2}$ for some $p>n$ satisfying the compatibility condition \eqref{compatibility}, there exists a unique classical solution to \eqref{System} on a maximal interval $(0,T_{\max})$, i.e. for any $0<T<T_{\max}$,
	\begin{equation*}
		(u,v)\in C([0,T];L^p(\Omega)^{m_1}\times L^p(\M)^{m_2}) \cap L^\infty(0,T;L^\infty(\Omega)^{m_1}\times \LM{\infty}^{m_2}),
	\end{equation*}
	\begin{equation*}
		u\in (C^{2,1}(\overline{\Omega}\times(\tau,T)))^{m_1}, \quad v \in (C^{2,1}(\M\times(\tau,T)))^{m_2} \quad \text{ for all } \, 0<\tau<T,
	\end{equation*}
	and the equations \eqref{System} satisfy pointwise.
	
	\medskip
	The following blow-up criterion holds
	\begin{equation}\label{BlowUpCriterion}
		T_{\max}<+\infty \quad \Longrightarrow \quad \limsup_{t\rightarrow T_{\max}} \left(\sumi\|u_i(t)\|_{\LO{\infty}} + \sumjj \|v_j(t)\|_{\LM{\infty}}\right) = +\infty.
	\end{equation}
	Moreover, if \eqref{A2} holds, then $(u(t),v(t))$ is non-negative provided the initial data $(u_0,v_0)$ is non-negative.
\end{theorem}

We will now prove some interpolation inequalities, which will be used in the sequel analysis.

\begin{lemma}\label{Lem:interpolation:Omega}
	For any $T>0$, and any $2<\kappa < \frac{2(n+2)}{n} $, there exists $C_{T,\kappa}>0$ such that
	\begin{equation}\label{Lem:interpolation:Omega:State}
		\|w\|^{\kappa}_{L^\kappa(Q_T)}\leq C_{T,\kappa}\|w\|^{\frac{2\kappa-(\kappa-2)n}{2}}_{L^\infty(0,T;L^2(\Omega))}\|w\|^{\frac{n(\kappa-2)}{2}}_{L^2(0,T;H^1(\Omega))}
	\end{equation}
	and
	\begin{equation}\label{Lem:interpolation:Omega:State2}
		\|w\|^{2}_{L^\kappa(Q_T)}\leq C_{T,\kappa}(\|w\|^{2}_{L^\infty(0,T;L^2(\Omega))} + \|w\|^{2}_{L^2(0,T;H^1(\Omega))}).
	\end{equation}
\end{lemma}

\begin{proof}
	The result of this lemma should be standard. However, since we cannot find an explicit reference, we give here a full proof for the sake of completeness. By using the Gagliardo-Nirenberg inequality, for any $\kappa>2$, we have 
	\begin{equation*}
		\|w\|^\kappa_{L^\kappa(\Omega)}\leq C_{T,\kappa}\|w\|^{\frac{2\kappa-(\kappa-2)n}{2}}_{L^2(\Omega)}\|w\|^{\frac{n(\kappa-2)}{2}}_{H^1(\Omega)}.
	\end{equation*} 
	Integrating over $(0,T)$, we have 
	\begin{equation*}
		\begin{split}
			\int_0^T \|w\|_{\LO{\kappa}}^\kappa &\leq C\int^T_0\|w\|^{\frac{2\kappa-(\kappa-2)n}{2}}_{L^2(\Omega)}\|w\|^{\frac{n(\kappa-2)}{2}}_{H^1(\Omega)}\\
			&\leq C\|w\|^{\frac{2\kappa-(\kappa-2)n}{2}}_{L^\infty(0,T;L^2(\Omega))} \int^T_0\|w\|^{\frac{n(\kappa-2)}{2}}_{H^1(\Omega)}.
		\end{split}
	\end{equation*}
	Due to the condition $\kappa < \frac{2(n+2)}{n}$, we have $\frac{n(\kappa-2)}{2}<  2$, and therefore we can use H\"{o}lder’s inequality to have
	\begin{equation*}
		\begin{split}
			\int^T_0\|w\|^{\frac{n(\kappa-2)}{2}}_{H^1(\Omega)}\leq C_T\|w\|^{\frac{n(\kappa-2)}{2}}_{L^2(0,T;H^1(\Omega))}.
		\end{split}
	\end{equation*}
	Inserting this into the previous inequality gives the desired interpolation inequality \eqref{Lem:interpolation:Omega:State}. 
	The estimate \eqref{Lem:interpolation:Omega:State2} can be obtained by using Young's inequality.
\end{proof}

\begin{lemma} \label{Lem:M}
	For any $T>0$, and any $2<\sigma < 2+ \frac{2}{n}$, there exists $C_{T}>0$ such that 
	\begin{equation}\label{Lem:M:State2}
		\begin{split}
			\| w\|_{\LQM{\sigma}}^{\frac{\sigma}{\sigma-1}}   
			&\leq C_{T}(\|w\|^{2}_{L^\infty(0,T;L^2(\Omega))} + \|w\|^{2}_{L^2(0,T;H^1(\Omega))} + 1).
		\end{split}
	\end{equation}
\end{lemma}

\begin{proof}
	Using the interpolation trace inequality ( see e.g. \cite[Proof of Theorem 1.5.1.10, page 41]{grisvard2011elliptic}), we  obtain
	\begin{align*}
		\intOM w^{\sigma} &\leq C \sigma\intO w^{\sigma-1} |\nabla w| + C \intO  w^{\sigma} \\
		&\leq C \sigma \Big(\intO |\nabla w|^2\Big)^{\frac{1}{2}} \Big(\intO w^{2\sigma -2}\Big)^{\frac{1}{2}}      + C \intO  w^{\sigma}.
	\end{align*}
	Integrating over $(0,T)$ and using Young's inequality, we have
	\begin{align*}
		\intQM w^{\sigma}  
		&\leq C \sigma  \intQ |\nabla w|^2 +  C \sigma \intQ w^{2\sigma -2}    + C \intQ  w^{\sigma}.
	\end{align*}
	Therefore,
	\begin{align*}
		\| w\|_{\LQM{\sigma}}^{\frac{2\sigma}{2\sigma-2}}  
		&\leq C \sigma^{\frac{2}{2\sigma-2}}   \|\nabla w\|_{\LQ{2}}^{\frac{4}{2\sigma-2}} +  C \sigma^{\frac{2}{2\sigma-2}} \|w\|_{\LQ{2\sigma-2}}^{2}    + C \|w\|_{\LQ{\sigma}}^{\frac{2\sigma}{2\sigma-2}}\\
		&\leq C \sigma^{\frac{2}{2\sigma-2}}   (\|\nabla w\|_{\LQ{2}}^{2} + C) +  C \sigma^{\frac{2}{2\sigma-2}} \|w\|_{\LQ{2\sigma-2}}^{2}    + C (\|w\|_{\LQ{\sigma}}^{2} + C)\\
		&\leq C (\|\nabla w\|_{\LQ{2}}^{2} + C) +  C  \|w\|_{\LQ{2\sigma-2}}^{2}    + C (\|w\|_{\LQ{\sigma}}^{2} + C).
	\end{align*}
	For the term $\|w\|_{\LQ{2\sigma-2}}^{2}$. Since $2< 2\sigma -2 <   \frac{2(n+2)}{n} $ due to $2<\sigma < 2+ \frac{2}{n} $, by using Lemma \ref{Lem:interpolation:Omega}, we have 
	\begin{align*}
		\|w\|^{2}_{L^{2\sigma-2}(Q_T)}\leq C_{T}(\|w\|^{2}_{L^\infty(0,T;L^2(\Omega))} + \|w\|^{2}_{L^2(0,T;H^1(\Omega))})
	\end{align*}
	and similarly
	\begin{align*}
		\|w\|^{2}_{L^{\sigma}(Q_T)}\leq C_{T}(\|w\|^{2}_{L^\infty(0,T;L^2(\Omega))} + \|w\|^{2}_{L^2(0,T;H^1(\Omega))})
	\end{align*}
	Thus, we have
	\begin{align*}
		\| w\|_{\LQM{\sigma}}^{\frac{\sigma}{\sigma-1}}   
		&\leq C_{T}(\|w\|^{2}_{L^\infty(0,T;L^2(\Omega))} + \|w\|^{2}_{L^2(0,T;H^1(\Omega))} + 1).
	\end{align*}
	This finish the proof of Lemma \ref{Lem:M}.	  
\end{proof}

\begin{lemma}\label{Lem:interpolation:M}
	For any $T>0$, and any $2<\xi <  \frac{2(n+1)}{n-1} $, there exists $C_{T,\xi}>0$ such that
	\begin{equation}\label{Lem:interpolation:M:State}
		\|w\|^{\xi}_{L^\xi(M_T)}\leq C_{T,\xi}\|w\|^{\frac{2\xi-(\xi-2)(n-1)}{2}}_{L^\infty(0,T;L^2(\M))}\|w\|^{\frac{(n-1)(\xi-2)}{2}}_{L^2(0,T;H^1(\M))}
	\end{equation}
	and
	\begin{equation}\label{Lem:interpolation:M:State2}
		\|w\|^{2}_{L^\xi(M_T)}\leq C_{T,\xi}(\|w\|^{2}_{L^\infty(0,T;L^2(\M))} + \|w\|^{2}_{L^2(0,T;H^1(\M))})
	\end{equation}
\end{lemma}

\begin{proof}
	By using the Gagliardo-Nirenberg inequality on the compact $(n-1)$-dimensional manifold $M$ without boundary, see e.g. \cite[Theorem 3.70]{aubin2012nonlinear}, we have for any $\xi>2$
	\begin{equation*}
		\|w\|^\xi_{L^\xi(\M)}\leq C_{T,\xi}\|w\|^{\frac{2\xi-(\xi-2)(n-1)}{2}}_{L^2(\M)}\|w\|^{\frac{(n-1)(\xi-2)}{2}}_{H^1(\M)}.
	\end{equation*}
	Therefore, we have
	\begin{equation*}
		\begin{split}
			\intQM w^\xi \,dxdt&\leq C\int^T_0\|w\|^{\frac{2\xi-(\xi-2)(n-1)}{2}}_{L^2(\M)}\|w\|^{\frac{(n-1)(\xi-2)}{2}}_{H^1(\M)}dt\\
			&\leq C\|w\|^{\frac{2\xi-(\xi-2)(n-1)}{2}}_{L^\infty(0,T;L^2(\M))}\int^T_0\|w\|^{\frac{(n-1)(\xi-2)}{2}}_{H^1(\M)}dt.
		\end{split}
	\end{equation*}
	It is noted that  $\frac{(n-1)(\xi-2)}{2}< 2$ due to $2<\xi< \frac{2(n+1)}{n-1}$, and therefore we can use H\"{o}lder’s inequality to have
	\begin{equation*}
		\begin{split}
			\int^T_0\|w\|^{\frac{(n-1)(\xi-2)}{2}}_{H^1(\M)}dt\leq C_T\|w\|^{\frac{(n-1)(\xi-2)}{2}}_{L^2(0,T;H^1(\M))}.
		\end{split}
	\end{equation*}
	Inserting this into the previous inequality gives the desired interpolation inequality \eqref{Lem:interpolation:M:State}. 
	The estimate \eqref{Lem:interpolation:M:State2} can be obtained by using Young's inequality.
\end{proof}

The following elementary lemma is from reference \cite{morgan2023global}, where the detailed proof is provided.
\begin{lemma}\label{Lem:key}
	Let $\{y_j\}_{j=1,\cdots , m_2}$ be a sequence of non-negative numbers. Assume that there is a constant $K>0$ such that, for any $\eps>0$, there exists $C_\eps>0$  such that if $k\in \{1,\cdots , m_2\}$ 
	\begin{equation*}
		y_k \leq C_\eps+ K\sum_{j=1}^{k-1}y_j + \eps\sumjj y_j,
	\end{equation*}
	where if $k=1$, the sum $\sum_{j=1}^{k-1}y_j$ is neglected. Then, there exists a constant $C$ independent of the sequence $\{y_j\}$ such that 
	\begin{equation}\label{bound_elementary}
		\sumjj y_j \leq C.
	\end{equation}
\end{lemma}

\section{Global existence} \label{Sec:GlobalExsitence}
\subsection{Moser iteration and bounds for time integration of solutions} 
We start with $L^\infty_tL^1_x$-bounds of solutions, which are proved in \cite{morgan2023global}.
\begin{lemma}\label{Lem:L^p:energy}
	Assume \eqref{A1}, \eqref{A2}, \eqref{A3} and \eqref{A4} with \eqref{Condition:Growth}. 
	It holds that
	\begin{align*}
		\|u_i\|_{L^{\infty}(0,T;\LO{1})} + \|u_i\|_{\LQM{1}} + \|v_j\|_{L^{\infty}(0,T;\LOM{1})}\leq C_T.
	\end{align*}
\end{lemma} 

Under the mass control condition \eqref{A3},  we now show that the solution satisfies the following properties $\int^t_0 u_i(x,s) \,ds \in \LQ{\infty}, ~ \int^t_0 u_i(x,s) \,ds \in \LQM{\infty} ~\text{and} ~ \int^t_0 v_j(x,s) \,ds \in \LQM{\infty}$. 
This will play a crucial role in the proof of the main result. 
\begin{lemma}\label{Lem:L_xL^1_t}
	Assume \eqref{A1}, \eqref{A2} and \eqref{A3}. Then, there exists a constant  $C_{T}>0$ such that 
	\begin{align}\label{star}
		\Big\|\int^t_0 u_i \,ds\Big\|_{\LQ{\infty}} + \Big\|\int^t_0 u_i \,ds\Big\|_{\LQM{\infty}} + \Big\|\int^t_0 v_j \,ds\Big\|_{\LQM{\infty}} \leq C_T 
	\end{align}
	for any $i=1,\cdots,m_1$ and $j=1,\cdots,m_2$.
\end{lemma}

\begin{proof}
	By summing the equation of \eqref{System} and using \eqref{A3}, we have 
	\begin{align}\label{Lem:L_xL^1_t:Sys1}
		\begin{cases}
			\pa_t \(\sumii u_i\) - \Delta  \left( \sumii d_i  u_i\right)  = \sumii F_i(u)\leq L(|u|+ 1), \vspace*{0.2cm} &\text{on } Q_T,\\
			\partial_{\nu} (\sumii d_i  u_i) = \sumii G_i(u,v)\leq L(|u| + |v| + 1), \vspace*{0.2cm} &\text{on } \M_T,\\
			\pa_t (\sumjj v_j) - \Delta_{\M} ( \sumjj \delta_j v_j) = \sumjj H_j(u, v)\leq L(|u| + |v| + 1), &\text{on } \M_T,  \vspace*{0.2cm}\\
			\sumii u_i(\cdot,0) = \sumii u_{i0} \; \text{ in } \Omega, \quad \sumjj v_j(\cdot,0) = \sumjj v_{j0}  \text{on } M.
		\end{cases}  
	\end{align}
	We define
	\begin{align*}
		\widetilde{W}(x,t): = \int^t_0 \Big( \sumii d_i  u_i(x,s)\Big) \,ds \quad\text{and}\quad
		\widetilde{Z}(x,t): = \int^t_0 \Big( \sumjj \delta_j  v_j(x,s)  \Big) \,ds. 
	\end{align*}
	By integrating equations \eqref{Lem:L_xL^1_t:Sys1}  with respect to time and using the condition \eqref{A3},  we obtain the system
	\begin{align}\label{System:aux}
		\begin{cases}
			\widetilde{A} \pa_t \widetilde{W} - \Delta \widetilde{W} \leq  \Lambda \widetilde{W} + Lt, \vspace*{0.15cm}\\
			\nabla \widetilde{W} \cdot \nu \leq  \Lambda \widetilde{W} +  \lambda\widetilde{Z} + Lt, \vspace*{0.15cm}\\
			\widetilde{B} \pa_t \widetilde{Z} - \Delta_{\M} \widetilde{Z}  \leq  \Lambda \widetilde{W} + \lambda \widetilde{Z}  + Lt, \vspace*{0.15cm}\\
			\widetilde{W} (x,0) = 0, \quad \widetilde{Z} (x,0) = 0,
		\end{cases}
	\end{align}
	where $\Lambda= \frac{|L|}{\min\{d_i\}}$, $\lambda= \frac{|L|}{\min\{\delta_j\}}$ and $\widetilde{A}$,  $\widetilde{B}$ are given by
	\begin{align*}
		0< \frac{1}{\max\{d_1,\cdots,d_{m_1}\}}:=\underline{a}   \leq \widetilde{A}(x,t) = \frac{\sumii u_i}{\sumii d_i  u_i } \leq \overline{a}=:\frac{1}{\min\{d_1,\cdots,d_{m_1}\}}
	\end{align*} 
	and  
	\begin{align*}
		0< \frac{1}{\max\{\delta_1,\cdots,\delta_{m_2}\}}:=\underline{b}\leq\widetilde{B}(x,t) = \frac{\sumjj v_j}{\sumjj \delta_j  v_j}\leq \overline{b}=:\frac{1}{\min\{\delta_1,\cdots,\delta_{m_2}\}}.
	\end{align*}

	The proof uses the ideas from \cite{alikakos1979lp}.  
	For any $k\geq1$,  multiplying the first equation of System \eqref{System:aux} by $\widetilde{W}^{2^k-1}$ and integrating over $\Omega$, and similarly,  multiplying the third equation of system \eqref{System:aux} by $\widetilde{Z}^{2^k-1}$ and integrating over $\M $, summing the results provide
	\begin{align*}
		& \intO \widetilde{A} (\partial_t\widetilde{W})\widetilde{W}^{2^{k-1}} + \intOM \widetilde{B}  (\partial_t\widetilde{Z})\widetilde{Z}^{2^{k-1}} + \frac{2^k-1}{2^{2k-2}} \intO   |\nabla (\widetilde{W}^{2^{k-1}})|^2 + \frac{2^k-1}{2^{2k-2}}  \intOM   |\nabla_{\M} (\widetilde{Z}^{2^{k-1}})|^2  \\
		& \leq  \Lambda\intO \widetilde{W}^{2^k} + Lt \intO  \widetilde{W}^{2^k-1} + \Lambda\intOM \widetilde{W}^{2^k}  + \lambda\intOM \widetilde{W}^{2^k-1}  \widetilde{Z} + Lt\intOM \widetilde{W}^{2^k-1}   \\ 
		&+  \lambda\intOM \widetilde{Z}^{2^k} +  Lt \intOM \widetilde{Z}^{2^k-1} + \Lambda \intOM  \widetilde{W} \widetilde{Z}^{2^k-1}.
	\end{align*}
	Thanks to the fact that $\partial_t \widetilde{W} \ge 0$ and $\partial_t \widetilde{Z} \ge 0$, we have
	\begin{equation*}
		\intO \widetilde{A} (\partial_t\widetilde{W})\widetilde{W}^{2^{k-1}} \ge \frac{\underline{a}}{2^k}\frac{d}{dt}\intO \widetilde{W}^{2^k}, \quad \text{ and } \quad \intOM \widetilde{B} (\partial_t\widetilde{Z})\widetilde{Z}^{2^{k-1}} \ge \frac{\underline{b}}{2^k}\frac{d}{dt}\intOM \widetilde{Z}^{2^k}.
	\end{equation*}
	For the some terms of the right hand side with the exponent $2^{k-1}$, we have the inequalities   
	\begin{align*}
		Lt \intO  \widetilde{W}^{2^k-1} + Lt\intOM \widetilde{W}^{2^k-1} +  Lt \intOM \widetilde{Z}^{2^k-1} 
		&\leq \frac{C t }{2^k}  + \frac{|L| (2^k -1) t}{2^k} \intO  \widetilde{W}^{2^k}  \\
		& \quad+ \frac{|L|(2^k -1)}{2^k} \intOM \widetilde{W}^{2^k}  
		+ \frac{|L| (2^k -1) t}{2^k} \intOM \widetilde{Z}^{2^k} 
	\end{align*}
	for some $C$ depending only on $L$, $M$ and $|\Omega|$, and
	\begin{align*}
		\lambda\intOM \widetilde{W}^{2^k-1}  \widetilde{Z} + 
		\Lambda \intOM  \widetilde{W} \widetilde{Z}^{2^k-1} &\leq \frac{\lambda + \Lambda(2^k -1)}{2^k} \intOM \widetilde{Z}^{2^k} + \frac{\Lambda+ \lambda(2^k -1)}{2^k} \intOM \widetilde{W}^{2^k}.
	\end{align*} 
	Noted that the above estimates gives 
	\begin{align}\label{Lem:L_xL^1_t:Proof1}
		\begin{split}
			&\frac{\underline{a}}{2^k} \frac{d}{dt}\intO   \widetilde{W}^{2^k} + \frac{\underline{b}}{2^k}\frac{d}{dt} \intOM  \widetilde{Z}^{2^k} + \frac{2^k-1}{2^{2k-2}} \intO   |\nabla (\widetilde{W}^{2^{k-1}})|^2 + \frac{2^k-1}{2^{2k-2}}  \intOM   |\nabla_{\M} (\widetilde{Z}^{2^{k-1}})|^2  \\
			& \leq \Big( \Lambda+  \frac{|L| (2^k -1) t}{2^k} \Big) \intO \widetilde{W}^{2^k} + \Big( \Lambda+ \frac{|L|(2^k -1)t + \lambda(2^k -1) + \Lambda}{2^k} \Big)\intOM \widetilde{W}^{2^k}   \\ 
			&\quad + \Big( \lambda+ \frac{|L| (2^k -1)  t + \lambda+ \Lambda(2^k -1)}{2^k} \Big)\intOM \widetilde{Z}^{2^k} + \frac{Ct}{2^k} \\
			& \leq  \left( \Lambda+ |L|t \right) \intO \widetilde{W}^{2^k} + \Big( \Lambda+ |L|t+ \lambda+ \frac{\Lambda}{2^k} \Big)\intOM \widetilde{W}^{2^k}  + \Big( \lambda+ |L|t  + \Lambda+ \frac{\lambda}{2^k}  \Big)\intOM \widetilde{Z}^{2^k} + \frac{Ct}{2^k}.
		\end{split}
	\end{align}
	Therefore, there exists $\alpha>0$ independent of $k$ such that
	\begin{align*}
		&\frac{1}{2}  \frac{d}{dt} \left(\intO \widetilde{W}^{2^k} + \intOM \widetilde{Z}^{2^k}\right) + \alpha \left(\intO   |\nabla (\widetilde{W}^{2^{k-1}})|^2 +   \intOM   |\nabla_{\M} (\widetilde{Z}^{2^{k-1}})|^2 \right)  \\
		& \leq  c_1 2^{k-1} \intO \widetilde{W}^{2^k} +  c_2  2^{k-1} \intOM \widetilde{W}^{2^k}  + c_3  2^{k-1} \intOM \widetilde{Z}^{2^k} + c_4,
	\end{align*}
	where $c_1= \frac{\Lambda + |L|T}{\min \{\underline{a}, \underline{b} \}} $, $c_2 = \frac{2\Lambda + |L|T + \lambda}{\min \{\underline{a}, \underline{b} \}}$, $c_3 = \frac{2\lambda + |L|T  + \Lambda}{\min \{\underline{a}, \underline{b} \}} $ and $c_4 = \frac{CT}{2\min \{\underline{a}, \underline{b} \}} $. 
	Denote by $W_k:= \widetilde{W}^{2^{k-1}}$, $Z_k:= \widetilde{Z}^{2^{k-1}}$ and $C_k: = \max\{c_1, c_2, c_3\} 2^{k-1}$ we get 
	\begin{align}\label{Lem:L_xL^1_t:ProofG}
		\begin{split}
			&\frac{1}{2}\frac{d}{dt} \left(\intO W_k^2 + \intOM Z_k^2 \right) + \alpha \left(\intO   |\nabla W_k|^2 +   \intOM   |\nabla_{\M} Z_k|^2 \right)  \\
			& \leq  C_k \intO W_k^2 +  C_k \intOM W_k^2 + C_k \intOM Z_k^2 + c_4.
		\end{split}
	\end{align}  
	Using Gagliardo-Nirenberg inequality, for any $\eps_1>0$, we have
	\begin{align}\label{Lem:L_xL^1_t:ProofG1}
		\intO |W_k|^2  \le \eps_1 \intO|\nabla W_k|^2  + C_{\Omega,n}\eps_1^{-n/2} \left(\intO |W_k|  \right)^{2} 
	\end{align}
	and
	\begin{align*}
		\intOM |Z_k|^2  \le \frac{\alpha}{2} \intOM |\nabla_{\M} Z_k|^2  + C_{\M,n,\alpha}  \left(\intOM |Z_k|  \right)^{2} 
	\end{align*} 
	for some $C_{\Omega,n}$ depending only on $\Omega$ and $n$, $C_{\M,n,\alpha}$ depending only on $\M$, $n$ and $\alpha$. Combining the inequality \eqref{Lem:L_xL^1_t:ProofG1} with the trace interpolation inequality, for any $\eps>0$,  
	\begin{align*}
		\begin{aligned}
			\intOM|W_k|^2  &\le \eps  \intO |\nabla W_k|^2  + \frac{C_{\Omega,n}}{\eps}\intO |W_k|^2 \\
			&\le \left(\eps + \frac{C_{\Omega,n}\eps_1}{\eps}\right) \intO |\nabla W_k|^2  + \frac{C_{\Omega,n}^2 \eps_1^{-n/2}}{\eps} \left(\intO |W_k| \right)^2.
		\end{aligned}
	\end{align*}
	Now, we choose 
	\begin{align*} 
		\eps = \frac{\alpha}{8 C_k} \quad  \text{ and}\quad  \eps_1 \le \frac{\alpha \eps}{8C_{\Omega,n} C_k} =  \left(\frac{\alpha}{8}\right)^2\frac{1}{C_{\Omega,n} C_k^{2}} .
	\end{align*}
	It is noted that $\eps_1 C_k \le \frac{\alpha}{4}$ if we choose $C_{\Omega,n}$ large enough. Therefore
	\begin{align*}
		C_k\intO |W_k|^2  + C_k\intOM  |W_k|^2  &\le \frac{\alpha}{2}\intO |\nabla W_k|^2  + C_{\Omega,n} \left(C_k^{n+1} + C_k^{n+2}\right)  \left(\intO |W_k| \right)^2.
	\end{align*} 
	Thus we have
	\begin{align}\label{Lem:L_xL^1_t:Proof2}
		\begin{split}
			\frac{1}{2}\frac{d}{dt} \left(\intO W_k^2 + \intOM Z_k^2 \right) &+ \frac{\alpha}{2} \left(\intO |\nabla W_k|^2 +   \intOM   |\nabla_{\M} Z_k|^2 \right)  \\
			& \leq  \hat{C}_k \left(\intO |W_k| +   \intOM |Z_k| \right)^2 + c_5,
		\end{split}
	\end{align} 
	where $\hat{C}_k \sim (2^{n+2})^{k}$. Adding both sides with $(\alpha/2)(\intO W_k^2 + \intOM Z_k^2)$ and using the Gagliardo-Nirenberg inequality again, we arrive at
	\begin{align*}
		& \frac{d}{dt} \left(\intO W_k^2 + \intOM Z_k^2 \right) + \alpha \left(\intO W_k^2 + \intOM Z_k^2\right)   \leq  \hat{C}_k \left(\intO |W_k| +   \intOM |Z_k| \right)^2 + c_6 
	\end{align*}
	with $\hat{C}_k \sim (2^{n+2})^{k}$.
	Let $y_k:= \intO W_k^2 + \intOM Z_k^2 $, which gives
	\begin{align*}
		& \frac{d}{dt} y_k + \alpha y_k   \leq  \hat{C}_k \left(\intO |W_k| +   \intOM |Z_k| \right)^2 + c_6.
	\end{align*}
	An application of the classical Gronwall inequality yields for any $t\in (0,T)$
	\begin{align*}
		\begin{aligned} 
			y_k(t) &\le e^{-\alpha t} y_k(0) + \hat{C}_k\int_0^te^{-\alpha(t-s)} \Big(\intO |W_k| +   \intOM |Z_k| \Big)^2 \,ds + \frac{c_6}{\alpha}  \\
			&\le  \frac{\hat{C}_k}{\alpha} \Big(\sup_{s\in (0,T)}\Big(\intO |W_k| +   \intOM |Z_k| \Big)\Big)^2 + \frac{c_6}{\alpha},
		\end{aligned}
	\end{align*}
	where we used $y_k(0) =0$. Replacing $y_k = \intO W_k^2 + \intOM Z_k^2 = \intO \widetilde{W}^{2^{k}} + \intOM \widetilde{Z}^{2^{k}} $ again, and taking the root of order $2^k$ gives
	\begin{align*} 
		&\Big(\sup_{t\in(0,t_0)} \Big(\intO \widetilde{W}^{2^{k}} (t) + \intOM \widetilde{Z}^{2^{k}}  (t) \Big) \Big)^{1/2^k}  \le   \Big( \frac{\hat{C}_k}{\alpha} \Big(\sup_{s\in (0,t_0)}\Big(\intO \widetilde{W}^{2^{k-1}} + \intOM \widetilde{Z}^{2^{k-1}} \Big)\Big)^2 + \frac{c_6}{\alpha} \Big)^{1/2^k}.
	\end{align*}
	By denoting $Y_k$ the left hand side of this inequality, we get
	\begin{align*} 
		Y_k \le  \left( \frac{\hat{C}_k}{\alpha}Y_{k-1}^{2^k} + \frac{c_6}{\alpha} \right)^{1/2^k}.
	\end{align*}
	Therefore, with $c_7 =   \alpha^{-1} + c_6 \alpha^{-1}$, we have
	\begin{align*} 
		&\max\{Y_k,  1\}  \le \max\{Y_{k-1},  1\} c_7^{1/2^k} \hat{C}_k^{1/2^k}.
	\end{align*}
	Thus, by $\hat{C}_k \sim (2^{n+2})^k$
	\begin{equation*} 
		\begin{aligned}
			\max\{Y_k,  1\}  
			&\le \max\{Y_0,  1\} \prod_{k=1}^{\infty}c_7^{1/2^k}\hat{C}_k^{1/2^k} 
			\le C_{T} c_7^{\sum_{k\ge 1}(1/2^k)}(2^{n+2})^{\sum_{k\ge 1}(k/2^k)}
			\le C_{T}. 
		\end{aligned}
	\end{equation*} 
	Letting $k\to \infty$, we obtain finally the estimate 
	\begin{align}\label{Lem:L_xL^1_t:Proof:infty}
		\|\widetilde{W} \|_{\LQ{\infty}}  + \|\widetilde{Z} \|_{\LQM{\infty}} \leq C_T.
	\end{align} 
	This already shows the desired bounds for the first and the third terms on left hand side of \eqref{star}. To deal with the second term, we integrate over $(0,T)$ for \eqref{Lem:L_xL^1_t:ProofG} to get
	\begin{align*}
		\begin{split}
			&\sup_{t\geq 0} \left(\intO W_k^2 + \intOM Z_k^2 \right) + 2\alpha \left(\intQ   |\nabla W_k|^2 +   \intQM   |\nabla_{\M} Z_k|^2 \right)  \\
			& \leq  2 C_k \intQ W_k^2 +  2 C_k \intQM W_k^2 + 2 C_k \intQM Z_k^2 + 2 c_4 T.
		\end{split}
	\end{align*}
	where $C_k \sim  2^{k-1} $. 
	Applying $$2 C_k    \intOM W_k^2 \leq \alpha \intO |\nabla W_k|^2  + \tilde{C}_k  \intO W_k^2, $$ 
	where $\tilde{C}_k \sim  2^{2(k-1)} $.
	Thus
	\begin{align}\label{Lem:L_xL^1_t:Proof3}
		\begin{split}
			\sup_{t\geq 0} \left(\intO \widetilde{W}^{2^{k}}  + \intOM \widetilde{Z}^{2^{k}} \right) &+ \alpha \left(\intQ   |\nabla \widetilde{W}^{2^{k-1}}|^2 +   \intQM   |\nabla_{\M} \widetilde{Z}^{2^{k}}|^2 \right)  \\
			& \leq  \hat{C}_k \intQ \widetilde{W}^{2^{k}}   + \hat{C}_k \intQM \widetilde{Z}^{2^{k}} + 2 c_4t,
		\end{split}
	\end{align}
	where $\tilde{C}_k \sim  2^{2(k-1)} $. Applying Lemma \ref{Lem:M}, for any $k>2$ and $2<\sigma < 2+ \frac{2}{n} $,  we have  
	\begin{align}\label{Lem:L_xL^1_t:Proof4} 
		\|\widetilde{W}^{2^{k}}\|^{\frac{\sigma}{\sigma-1}}_{L^\sigma(M_T)}\leq C_{T} \Big(\|\widetilde{W}^{2^{k}}\|^{2}_{L^\infty(0,T;L^2(\Omega))} + \|\widetilde{W}^{2^{k}}\|^{2}_{L^2(0,T;H^1(\Omega))} +1\Big),
	\end{align} 
	where $1 + \frac{n}{n+2} < \frac{\sigma}{\sigma-1}<2$.	Applying Lemma \ref{Lem:interpolation:Omega}, for any $\varrho > 2$ and $2<\varrho <  \frac{2(n+2)}{n}$,  we have 
	\begin{align} \label{Lem:L_xL^1_t:Proof5}
		\|\widetilde{W}^{2^{k}}\|^{2}_{L^\varrho(Q_T)}\leq C_{T}(\|\widetilde{W}^{2^{k}}\|^{2}_{L^\infty(0,T;L^2(\Omega))} + \|\widetilde{W}^{2^{k}}\|^{2}_{L^2(0,T;H^1(\Omega))} ).
	\end{align} 
	
	Applying Lemma \ref{Lem:interpolation:M}, for any $\xi>2$ and $2<\xi <  \frac{2(n+1)}{n-1}$,  we have 
	\begin{align}\label{Lem:L_xL^1_t:Proof6} 
		\|\widetilde{Z}^{2^{k}}\|^{2}_{L^\xi(M_T)}\leq C_{T}(\|\widetilde{Z}^{2^{k}}\|^{2}_{L^\infty(0,T;L^2(\M))} + \|\widetilde{Z}^{2^{k}}\|^{2}_{L^2(0,T;H^1(\M))} ).
	\end{align} 
	By choosing $\varsigma := \min\{ \varrho,\sigma,\xi \}= \frac{2(n+1)}{n}$ for $n\geq 4$, combining with \eqref{Lem:L_xL^1_t:Proof4}, \eqref{Lem:L_xL^1_t:Proof5} and \eqref{Lem:L_xL^1_t:Proof6}, we have
	\begin{align*}
		\|\widetilde{W}^{2^{k}}\|_{L^\varsigma(Q_T)}^{\frac{\sigma}{\sigma-1}} &+ \|\widetilde{W}^{2^{k}}\|_{L^\varsigma(M_T)}^{\frac{\sigma}{\sigma-1}} +  \|\widetilde{Z}^{2^{k}}\|_{L^\varsigma(M_T)}^{\frac{\sigma}{\sigma-1}} \\
		&\leq C_{T} \Big(\|\widetilde{W}^{2^{k}}\|^{2}_{L^\infty(0,T;L^2(\Omega))} + \|\widetilde{W}^{2^{k}}\|^{2}_{L^2(0,T;H^1(\Omega))} +1\Big) \\
		&\quad + C_{T}(\|\widetilde{Z}^{2^{k}}\|^{2}_{L^\infty(0,T;L^2(\M))} + \|\widetilde{Z}^{2^{k}}\|^{2}_{L^2(0,T;H^1(\M))} ).
	\end{align*}
	where we used $1+\frac{n}{n+2} < \frac{\sigma}{\sigma-1}<2$ due to $2<\sigma < 2+ \frac{2}{n} $.
	Combining with the inequality \eqref{Lem:L_xL^1_t:Proof3}, we have
	\begin{align*}
		\begin{split}
			&\|\widetilde{W}^{2^{k}}\|_{L^\varsigma(Q_T)}^{\frac{\sigma}{\sigma-1}} + \|\widetilde{W}^{2^{k}}\|_{L^\varsigma(M_T)}^{\frac{\sigma}{\sigma-1}} +  \|\widetilde{Z}^{2^{k}}\|_{L^\varsigma(M_T)}^{\frac{\sigma}{\sigma-1}}  
			\leq \hat{C}_k  \|\widetilde{W}^{2^{k}}\|_{L^1(Q_T)}   + \hat{C}_k \|\widetilde{Z}^{2^{k}}\|_{L^1(Q_T)}  +   C_T.
		\end{split}
	\end{align*}
	Therefore, we have  
	\begin{align*}
		\begin{split} 
			&\Big(\|\widetilde{W}\|_{L^{2^{k}\varsigma}(Q_T)} +  \|\widetilde{W}\|_{L^{2^{k}\varsigma}(M_T)} + \|\widetilde{Z}\|_{L^{2^{k}\varsigma}(M_T)} \Big)^{2^{k}\frac{\sigma}{\sigma-1}} \\
			&\leq  \hat{C}_k \Big(\|\widetilde{W}\|^{2^{k}}_{L^{2^{k}}(Q_T)}  + \|\widetilde{W}\|^{2^{k}}_{L^{2^{k}}(M_T)} + \|\widetilde{Z}\|^{2^{k}}_{L^{2^{k}}(M_T)}\Big)  + C_T\\
			&\leq  \hat{C}_k \Big(\|\widetilde{W}\|_{L^{2^{k-1}\varsigma}(Q_T)}^{2^{k}}  + \|\widetilde{W}\|_{L^{2^{k-1}\varsigma}(M_T)}^{2^{k}} + \|\widetilde{Z}\|_{L^{2^{k-1}\varsigma}(M_T)}^{2^{k}} \Big)   + C_T.
		\end{split}
	\end{align*}
	Taking the root of order $2^k\frac{\sigma}{\sigma-1}$ and denoting $Q_k:= \|\widetilde{W}\|_{L^{2^{k}\varsigma}(Q_T)} +  \|\widetilde{W}\|_{L^{2^{k}\varsigma}(M_T)} + \|\widetilde{Z}\|_{L^{2^{k}\varsigma}(M_T)}$, we have
	\begin{align*}
			Q_k  \leq  \Big(\hat{C}_k Q_{k-1}^{2^k}    + C_T \Big)^{\frac{\sigma-1}{2^k\sigma}} \leq  \Big(\hat{C}_k Q_{k-1}^{2^k}    + C_T \Big)^{\frac{1}{2^k}}.
	\end{align*} 
	Therefore,  we have
	\begin{align*} 
		&\max\{Q_k,  1\}  \le \max\{Q_{k-1}, 1\} C_T^{\frac{1}{2^k}} \hat{C}_k^{\frac{1}{2^k}}.
	\end{align*}
	Thus, by $\hat{C}_k \sim 2^{2(k-1)}$
	\begin{equation*} 
		\begin{aligned}
			\max\{Q_k, 1\}  
			&\le \max\{Q_0, 1\} \prod_{k=1}^{\infty}C_T^{\frac{1}{2^k}} \hat{C}_k^{\frac{1}{2^k}} 
			\le C_{T} C_T^{\sum_{k\ge 1}(1/2^k)}(2^{2})^{\sum_{k\ge 1}(k/2^k)}
			\le C_{T}. 
		\end{aligned}
	\end{equation*} 
	Letting $k\to \infty$, we obtain finally the estimate
	\begin{align*} 
		\|\widetilde{W}\|_{\LQM{\infty}}\leq C_T.
	\end{align*}
	From this and  \ref{Lem:L_xL^1_t:Proof:infty}, we obtain, thanks to non-negativity of solutions, the desired bounds
	\begin{align*}
		\Big\|\int^t_0 u_i \Big\|_{\LQ{\infty}} +  \Big\|\int^t_0  u_i \Big\|_{\LQM{\infty}}  + \Big\|\int^t_0   v_j \Big\|_{\LQM{\infty}} \leq C_T
	\end{align*}
	for all $i=1,\ldots, m_1$ and $j=1,\ldots, m_2$.  
\end{proof}

\subsection{Duality method} 
We start with some useful estimates which are proved in \cite{morgan2023global}.
\begin{lemma}[\cite{morgan2023global}, Lemma 2.4]
	Moreover, for any positive integer $p\geq 2$ and any constant $\eps>0$, there exists $C_{p,\eps,T}>0$ such that  
	\begin{equation}\label{Lem:L^p:energy:State1}
		\sumii \left(\norm{u_i}_{\LQ{p-1+\rO}}^{p-1+\rO} + \norm{u_i}_{\LS{p-1+\rM}}^{p-1+\rM} \right)\leq C_{p,\eps,T}  + \eps\sumjj \norm{v_j}_{\LS{p-1+\rM}}^{p-1+\rM}.
	\end{equation}
	Consequently, for any $1<p<\infty$ and any $\eps>0$, there exists a constant $C_{p,\eps,T}>0$ such that
	\begin{align}\label{Lem:L^p:energy:State2}
		\sumii  \left(\|u_i\|_{\LQ{p}}+\|u_i\|_{\LS{p}}\right)\leq C_{p,\eps,T} + \eps\sumjj \|v_j\|_{\LS{p}}.
	\end{align} 
\end{lemma}
\begin{lemma}[\cite{morgan2023global}, Lemma 3.2]\label{Lem:DualProblem}
	Assume that $0\leq\tau<T$, $1<p<\infty$  and $ \theta_{\M}\in L^{p}(\M_{\tau,T})$. Let $\Psi_j $ be the solution to 
	\begin{align}\label{System_Dual}
		\begin{cases}
			\pa_t \Psi_j  + \Delta \Psi_j  = 0, & \text{in } Q_{\tau,T}\\
			\pa_t \Psi_j  + \delta_j \Delta_{\M}\Psi_j  = - \theta_{\M}, & \text{on } \M_{\tau,T}\\ 
			\Psi_j (x,T) = 0, & \text{in } \overline{\Omega}.
		\end{cases}
	\end{align}  
	Then, we have the estimate
	\begin{align*}
		\|\Psi_j\|_{W^{2,1}_p (\M_{\tau,T})} + \||\Psi_j |_{L^{\infty}(\tau,T)}\|_{\LOM{p}} \leq C_{T-\tau} \| \theta_{\M}\|_{L^{p}(\M_{\tau,T})} 
	\end{align*}
	and 
	\begin{align*}
		\|\Psi_j \|_{W^{2,1}_{p+\xi} (Q_{\tau,T})}& + \|\pa_{\nu} \Psi_j \|_{L^{p+ \xi}(Q_{\tau,T})} + \||\Psi_j |_{L^{\infty}(\tau,T)}\|_{\LO{p}}  \leq C_{T-\tau} \| \theta_{\M}\|_{L^{p}(\M_{\tau,T})},
	\end{align*}	
	where $\xi\leq \frac{p}{n+1}$. 
	Consequently, 
	\begin{align*} 
		\|\phi\|_{L^{q^{\dag}}(Q_{\tau,T})} + \|\phi\|_{L^{q^*}(\M_{\tau,T})} \leq C_{T-\tau}\|\psi\|_{L^{p}(\M_{\tau,T})}
	\end{align*}
	where
	\begin{align*} 
		q^{\dag} = 
		\begin{cases}
			\frac{(n+2)p}{n+2-2p} &\text{ if } p < \frac{n+1}{2},\\
			<+\infty \text{ abitrary } &\text{ if } p = \frac{n+1}{2},\\
			+\infty &\text{ if } p > \frac{n+1}{2}.
		\end{cases}\quad
		\text{ and }\quad
		q^* = \begin{cases}
			\frac{(n+1)p}{n+1-2p} &\text{ if } p < \frac{n+1}{2},\\
			<+\infty \text{ abitrary } &\text{ if } p= \frac{n+1}{2},\\
			+\infty &\text{ if } p > \frac{n+1}{2}.
		\end{cases}
	\end{align*} 
	Moreover, if $\theta_{\M} \geq 0$ a.e. in $M_{\tau,T}$, then $\Psi_j\geq 0$. 
\end{lemma}  

\begin{proof}
	The detailed proof can be found in \cite{morgan2023global}. Here, we just give the prove of 
	\begin{align*}
		\||\Psi_j |_{L^{\infty}(\tau,T)}\|_{\LOM{p}} \leq C_{T-\tau} \| \theta_{\M}\|_{L^{p}(\M_{\tau,T})} ,  \quad\text{and}\quad
		\||\Psi_j |_{L^{\infty}(\tau,T)}\|_{\LO{p}}  \leq C_{T-\tau} \| \theta_{\M}\|_{L^{p}(\M_{\tau,T})}.
	\end{align*}
	Indeed, if we apply the Sobolev embedding theorem, then there exists $C_{T-\tau}>0$ ($\tau\geq0$), such that
	\begin{align*}
		|\Psi_j |_{L^{\infty}(\tau,T)} \leq C_{T-\tau} |\Psi_j |_{W^{1,p}(\tau,T)}.
	\end{align*}
	Thus,
	\begin{align*}
		\||\Psi_j |_{L^{\infty}(\tau,T)}\|_{\LOM{p}} \leq  C_{T-\tau} \|\Psi_j\|_{W^{2,1}_p (\M_{\tau,T})} \leq C_{T-\tau} \| \theta_{\M}\|_{L^{p}(\M_{\tau,T})}. 
	\end{align*}
	Similar, we have 
	\begin{align*} 
		\||\Psi_j |_{L^{\infty}(\tau,T)}\|_{\LO{p}}  \leq C_{T-\tau} \| \theta_{\M}\|_{L^{p}(\M_{\tau,T})}.
	\end{align*}
\end{proof}

Next, by using the condition \eqref{A4}, Lemma \ref{Lem:L_xL^1_t} and the duality method, we are able to show the following Lemma, which plays a crucial role in the proof of the global existence.
\begin{lemma}\label{Lem:Dual:v}
	Assume \eqref{A1}, \eqref{A2}, \eqref{A3} and \eqref{A4} with \eqref{Condition:Growth}. For any $j\in \{1,\cdots,m_2\}$, it holds that
	\begin{align}\label{Lem:Dual:v:State}
		\begin{split}
			\|v_j\|_{\LQM{q}}       
			\leq C &+ \eps\Big( \sumii\|u_i\|_{\LQ{q}}  + \sumii \|u_i\|_{\LQM{q}} + \sumjj \|v_j\|_{\LQM{q}}\Big)    \\
			&+ C \sum_{k=1}^{j-1} \|v_k\|_{\LQM{q}},
		\end{split}
	\end{align} 
	where the last sum is neglected if $j=1$.
\end{lemma}
\begin{remark}
	Since $1+\frac{2}{n} < 1+ \frac{4}{n+1}$ for $n\geq 4$, the result has been obtained in \cite{morgan2023global} if $1\leq \mu_{\M} \leq \rO<1+\frac{2}{n}$. Therefore, we may focus on $\rO \leq \mu_{\M} <2$. 
\end{remark}

\begin{proof} 
	Let $1<p<\infty$ be chosen later, $q = p/(p-1)$ be the H\"older conjugate exponent of $p$, $0\le \theta_M\in L^p(M_T)$ be arbitrary, and $\Psi_j$ be the solution to \eqref{System_Dual}. 
	Using integration by parts, we have
	\begin{align}\label{Lem:Dual:v:Proof1}
		\begin{split}
			\intQM v_j  \theta_{\M} & = \intQM v_j \left( -\pa_t \Psi_j  - \delta_j \Delta_{\M} \Psi_j  \right) \\
			& = \intQM - \pa_t (v_j\Psi_j ) + \intQM \Psi_j  \left(  \pa_t v_j  - \delta_j \Delta_{\M} v_j \right)\\
			& = \intOM  v_{j0}\Psi_{j}(0) + \intQM \Psi_j  H_j(u,v).
		\end{split}
	\end{align} 
	With the application of H\"{o}lder's inequality and Lemma \ref{Lem:DualProblem}, it follows that 
	\begin{equation} \label{Lem:Dual:v:Proof2}
		\intOM  v_{j0}\Psi_{j}(0)  \leq C\|v_{j,0}\|_{\LOM{q}}\|\Psi_{j}(0)\|_{\LOM{p}} \leq C_T\| \theta_{\M}\|_{\LQM{p}} .
	\end{equation}
	Using the assumption  \eqref{A4}, for any $k = 1,\cdots , m_2$, we have  
	\begin{equation*}
		\sumii a_{(j+m_1)i}G_i(u,v) + \sum_{k=1}^j a_{(j+m_1)(k+m_1)}H_k(u,v) \leq K_1 \Big(|u|^{\mM} + |v|^{\mM} + 1 \Big),
	\end{equation*}
	which implies, 
	for all $j=1,\cdots ,  m_2$,
	\begin{equation*}\label{Lem:Dual:v:Proof3}
		\begin{aligned}
			H_j(u,v)
			& \leq -\sumii \frac{a_{(j+m_1)i}}{a_{(j+m_1)(j+m_1)}}G_i(u,v) - \sum_{k=1}^{j-1} \frac{a_{(j+m_1)(k+m_1)}}{a_{(j+m_1)(j+m_1)}} H_k(u,v)\\
			&\qquad  + \frac{K_1}{a_{(j+m_1)(j+m_1)}}\Big(|u|^{\mM} + |v|^{\mM} + 1 \Big).
		\end{aligned}
	\end{equation*}
	It follows from this and \eqref{Lem:Dual:v:Proof1}--\eqref{Lem:Dual:v:Proof2}, we arrive at the inequality
	\begin{align}\label{Lem:Dual:v:Proof4}
		\begin{aligned}
			\intQM v_j  \theta_{\M}    
			& \leq  C\| \theta_{\M}\|_{\LQM{p}} - \sumii \frac{a_{(j+m_1)i}}{a_{(j+m_1)(j+m_1)}} \intQM \Psi_j   G_i(u,v) \\
			&\quad - \sum_{k=1}^{j-1} \frac{a_{(j+m_1)(k+m_1)}}{a_{(j+m_1)(j+m_1)}} \intQM \Psi_j    H_k(u,v) \\
			&\quad +\frac{K_1}{a_{(j+m_1)(j+m_1)}} \intQM \Psi_j   \Big(|u|^{\mM} + |v|^{\mM} + 1 \Big) \\
			&= : C\| \theta_{\M}\|_{\LQM{p}} + \sum^3_i I_i.
		\end{aligned}
	\end{align}  
	
	Next, we will separately estimate the terms $I_i$ ($i=1,2,3$) on the right hand side of the above inequality \eqref{Lem:Dual:v:Proof4}.
	
	\medskip
	\noindent {\textbf{Estimate of $I_{1}$: }} First, we estimate the term $I_1$,
	\begin{align}\label{Lem:Dual:v:Proof}
		\begin{split}
			&	I_{1}:= - \sumii \frac{a_{(j+m_1)i}}{a_{(j+m_1)(j+m_1)}} \intQM \Psi_j   G_i(u,v) 
			=  - \sumii \frac{a_{(j+m_1)i}}{a_{(j+m_1)(j+m_1)}} \intQM \Psi_j   (d_i\nabla u_i \cdot \nu)\\
			& =  - \sumii \frac{a_{(j+m_1)i}}{a_{(j+m_1)(j+m_1)}} \Big(\intQ d_i \Delta u_i \Psi_j  + \intQM  d_i u_i\nabla\Psi_j    \cdot \nu - \intQ d_i  u_i \Delta \Psi_j  \Big)\\
			& =  - \sumii \frac{a_{(j+m_1)i}}{a_{(j+m_1)(j+m_1)}} \Big(\intQ \left(\pa_t u_i - F_i(u)\right) \Psi_j  + \intQM  d_i u_i\nabla\Psi_j    \cdot \nu \Big)\\
			&\quad +\sumii \frac{a_{(j+m_1)i}}{a_{(j+m_1)(j+m_1)}}  \intQ d_i  u_i \Delta \Psi_j  \\
			& =   \sumii \frac{a_{(j+m_1)i}}{a_{(j+m_1)(j+m_1)}}  \intO u_{i0}\Psi_{j}(0) +  \sumii \frac{a_{(j+m_1)i}}{a_{(j+m_1)(j+m_1)}} \intQ  F_i(u) \Psi_j  \\
			&\quad - \sumii \frac{a_{(j+m_1)i}}{a_{(j+m_1)(j+m_1)}}  \intQM  d_i u_i\nabla\Psi_j    \cdot \nu
			+  \sumii \frac{a_{(j+m_1)i}}{a_{(j+m_1)(j+m_1)}}  \intQ    u_i (\pa_t \Psi_j  + d_i\Delta \Psi_j  )\\
			& =: \sum^4_{k=1} I_{1k}.
		\end{split}
	\end{align}
	
	In order to estimate $I_1$, we need to estimate $I_{1k}$, $k=1,\cdots,4$.
	\begin{itemize}
		\item Estimate of $I_{11}$: By applying H\"{o}lder's inequality and using Lemma  \ref{Lem:DualProblem}, we obtain  the estimates
		\begin{align*}
			I_{11} : &=    \sumii \frac{a_{(j+m_1)i}}{a_{(j+m_1)(j+m_1)}}  \intO u_{i0}\Psi_{j}(0) 
			\leq C\|u_{i,0}\|_{\LOM{q}}\|\Phi_{j}(0)\|_{\LOM{p}} \leq C_T\|\theta_M \|_{\LQM{p}}.  
		\end{align*} 
		\item Estimate of $I_{12}$:     
		Based on the condition \eqref{A4}, we have $			\sumii a_{(j+m_1)i} F_i(u) \leq K_1 \Big(|u|^{\rO} + 1\Big)$, and thus, we obtain
		\begin{equation*} 
			\begin{split}
				I_{12}:&= \sumii \frac{a_{(j+m_1)i}}{a_{(j+m_1)(j+m_1)}} \intQ  F_i(u) \Psi_j  \leq C \intQ \Psi_j \Big(|u|^{\rO} + 1\Big)\\ 
				&\leq C\intQ  \Psi_j   + C \sumii  \intQ  |u_i|^{\rO} \Psi_j    \\ 
				&\leq   C_T\| \theta_{\M}\|_{\LQM{p}}   +   \sumii \intO |\Psi_j |_{L^{\infty}(0,T)} \int^T_0 |u_i|^{\rO} \,dsdx \\
				&\leq  C_T\| \theta_{\M}\|_{\LQM{p}}   + C \sumii\||\Psi_j |_{L^{\infty}(0,T)}\|_{\LO{p}} \Big(\intO  \Big(\int^T_0 |u_i|^{\rO}\,ds\Big)^q dx\Big)^{\frac{1}{q}}\\
				&\leq  C_T\| \theta_{\M}\|_{\LQM{p}}   + C \sumii\| \theta_{\M}\|_{\LQM{p}} \Big(\intO  \Big(\int^T_0 |u_i|^{\rO}\,ds\Big)^q dx\Big)^{\frac{1}{q}},  
			\end{split}
		\end{equation*} 
		where $q=\frac{p}{p-1}$ and  we used the H\"{o}lder's inequality together with Lemma  \ref{Lem:DualProblem}.
		For $\rO= 1$, by applying Lemmas \ref{Lem:L_xL^1_t} and \ref{Lem:DualProblem}, we get
		\begin{align*}
			\begin{split}
				I_{12} \leq  C\| \theta_{\M}\|_{\LQM{p}}.  
			\end{split}
		\end{align*} 
		Next, we consider $1<\rO< 1+ \frac{2}{n}$,  take  $k= p-\mu_{\M}(p-1) $, it noted that $\frac{\rO-k}{1-k} \leq \frac{\mu_{\M} -k}{1-k} =q=\frac{p}{p-1}$ and $ (1-k) q =  (1-k) \frac{p}{p-1}  <1$. 
		\begin{equation*} 
			\begin{split}
				&  \Big(\intO  \Big(\int^T_0 |u_i|^{\rO}\,ds\Big)^q dx\Big)^{\frac{1}{q}} \leq  \Big(\intO  \Big(\int^T_0 |u_i|^{\rO-k} |u_i|^k\,ds\Big)^q dx\Big)^{\frac{1}{q}} \\
				&\leq  \Big(\intO  \Big(\int^T_0 |u_i| \,ds\Big)^{kq}  \(\int^T_0 |u_i|^{\frac{\rO-k}{1-k}}  \,ds\)^{q(1-k)} \,dx\Big)^{\frac{1}{q}}\\
				&\leq   \Big\|\int^T_0 |u_i| \,ds\Big\|_{\LO{\infty}}^{k}  \Big(\intO  \int^T_0 |u_i|^{\frac{\rO-k}{1-k}}  \,ds \,dx\Big)^{1-k} |\Omega|^{\frac{1-(1-k)q}{q}}\\
				&\le C_T\|u_i\|_{\LQ{q}}^{1-k} 
			\end{split}
		\end{equation*}  
		thanks to $\frac{\rO - k}{1-k} \le q$. Combining the above estimates and using H\"{o}lder's inequality, we  obtain 
		\begin{align}\label{Lem:Dual:v:Proof5:2}
			I_{12} 
			\leq    \Big(C_{\eps,T} +  \frac{\eps}{2} \sumii\|u_i\|_{\LQ{q}}  \Big)\| \theta_{\M}\|_{\LQM{p}}.  
		\end{align}

		\item Estimate of $I_{13}$:  Using H\"{o}lder's inequality together with the estimates from Lemma \ref{Lem:DualProblem}, we obtain
		\begin{align*}
			I_{13}: &= - \sumii\frac{a_{(j+m_1)i}}{a_{(j+m_1)(j+m_1)}}\intQM d_i u_i\nabla\Psi_j \cdot\nu \leq C\sumii\|u_i\|_{\LQM{\frac{p+\xi}{p+\xi-1}}}\|\partial_{\nu}\Psi_j \|_{\LQM{p+\xi}}\\
			&\leq C\sumii\|u_i\|_{\LQM{\frac{p+\xi}{p+\xi-1}}}\| \theta_{\M}\|_{\LQM{p}} \leq \Big( \frac{\eps}{2} \sumii\|u_i\|_{\LQM{q}} + C_{\eps,T}\Big)\| \theta_{\M}\|_{\LQM{p}},
		\end{align*} 
		where $\xi $ from Lemma \ref{Lem:DualProblem}.
		
		\item Estimate of $I_{14}$:  Applying the estimates of Lemma \ref{Lem:DualProblem} together with H\"{o}lder's inequality, we  obtain
		\begin{align*}\label{Lem:Dual:v:Proof5:4}
			I_{14} : &= \sumii \frac{a_{(j+m_1)i}}{a_{(j+m_1)(j+m_1)}} \intQ u_i (\pa_t \Psi_j  + d_i\Delta \Psi_j  ) \\
			& \leq (C_{\eps,T} +  \frac{\eps}{2} \sumii\|u_i\|_{\LQ{q}} )\| \theta_{\M}\|_{\LQM{p}}.
		\end{align*}  
	\end{itemize}
	Combing the above estimates gives  
	\begin{align*}
			I_{1}  \leq  \Big( \frac{\eps}{2}  \sumii\|u_i\|_{\LQM{q}} + \eps \sumii\|u_i\|_{\LQ{q}} + C_{\eps,T} \Big) \| \theta_{\M}\|_{\LQM{p}}.  
	\end{align*} 
	
	\medskip
	\noindent{\textbf{Estimate of $I_{2}$: }} Since
	\begin{align*}
			I_{2}:&= - \sum_{k=1}^{j-1} \frac{a_{(j+m_1)(k+m_1)}}{a_{(j+m_1)(j+m_1)}} \intQM \Psi_j    H_k(u,v) \\
			& = - \sum_{k=1}^{j-1} \frac{a_{(j+m_1)(k+m_1)}}{a_{(j+m_1)(j+m_1)}} \intQM \Psi_j  \left(\pa_t v_k - \delta_k \Delta_{\M} v_k\right)\\
			& = \intOM v_{k0} \Psi_{j}(0) + \sum_{k=1}^{j-1} \frac{a_{(j+m_1)(k+m_1)}}{a_{(j+m_1)(j+m_1)}} \intQM v_k  \left(\pa_t \Psi_j  + \delta_k \Delta \Psi_j \right) \\
			& \leq \Big(C_T + C_T \sum_{k=1}^{j-1}  \|v_k\|_{\LQM{q}} \Big) \| \theta_{\M}\|_{\LQM{p}}, 
	\end{align*}
	where we used the estimates of Lemma \ref{Lem:DualProblem} and H\"{o}lder's inequality.
	
	\medskip
	\noindent{\textbf{Estimate of $I_{3}$: }} Since
	\begin{align*} 
			& I_{3}:= \frac{K_1}{a_{(j+m_1)(j+m_1)}} \intQM \Psi_j   \Big(|u|^{\mM} + |v|^{\mM} + 1 \Big)\\ 
			&\leq C \| \theta_{\M}\|_{\LQM{p}} + C \sumii \intOM | \Psi_j |_{L^{\infty}(0,T)} \int^T_0 |u_i|^{\mM}  \,dsdx \\
			&\quad + C \sumjj \intOM | \Psi_j |_{L^{\infty}(0,T)} \int^T_0  |v_j|^{\mM} \,dsdx \\
			&\leq   C \| \theta_{\M}\|_{\LQM{p}} + C\|| \Psi_j |_{L^{\infty}(0,T)}\|_{\LOM{p}}  \sumii \Big(\intOM  \Big(\int^T_0  |u_i|^{\mM} \,ds\Big)^q dx\Big)^{\frac{1}{q}}  \\
			&\quad+  C\|| \Psi_j |_{L^{\infty}(0,T)}\|_{\LOM{p}}   \sumjj  \Big(\intOM  \Big(\int^T_0   |v_j|^{\mM} \,ds\Big)^q dx\Big)^{\frac{1}{q}},
	\end{align*}
	where $q=\frac{p}{p-1}$. 
	For $\mM= 1$, with the application of  Lemmas \ref{Lem:L_xL^1_t} and \ref{Lem:DualProblem},  we directly get 
	\begin{align*}
		\begin{split}
			I_{3} \leq  C\| \theta_{\M}\|_{\LQM{p}}.  
		\end{split}
	\end{align*}
	Next, we consider  $1<\mM<2$,  take  $k= p-\mM(p-1) $, it is noted that $\frac{\mM-k}{1-k}  =q=\frac{p}{p-1}$ and $ (1-k) q =  (1-k) \frac{p}{p-1}  <1$ 
	\begin{equation} \label{Lem:Dual:v:Proof:1}
		\begin{split}
			\Big(\intOM  \Big(\int^T_0  u_i^{\mM} \,ds\Big)^q dx\Big)^{\frac{1}{q}}   & \leq   \Big(\intOM  \Big(\int^T_0 u_i^{\mM-k} u_i^k\,ds\Big)^q dx\Big)^{\frac{1}{q}}   \\
			&\leq   \Big(\intOM  \Big(\int^T_0 u_i \,ds\Big)^{kq}  \(\int^T_0 u_i^{\frac{\mM-k}{1-k}}  \,ds\)^{q(1-k)} \,dx\Big)^{\frac{1}{q}}\\ 
			&\leq    \Big\|\int^T_0 u_i \,ds\Big\|_{\LOM{\infty}}^{k}  \Big(\intOM  \int^T_0 u_i^{q}  \,dsdx\Big)^{1-k} |\M|^{\frac{1-(1-k)q}{q}} \\ 
			&\leq C_T  \|u_i\|_{\LQM{q}}^{1-k},  
		\end{split}
	\end{equation} 
	where we used Lemma \ref{Lem:L_xL^1_t}. 
	Similarly, we arrive at 
	\begin{equation} \label{Lem:Dual:v:Proof:2}
		\begin{split}
			\Big(\intOM  \Big(\int^T_0  v_j^{\mM} \,ds\Big)^q dx\Big)^{\frac{1}{q}}   &  \leq C_T  \|v_j\|_{\LQM{q}}^{1-k}.  
		\end{split}
	\end{equation} 
	Due to the fact that $k  <1$, combining the above estimates,  and then applying  Young's inequality, we obtain  
	\begin{equation}\label{Lem:Dual:v:Proof7}
		\begin{split}
			&I_{3}   \leq \Big( C_{\eps,T}+  \frac{\eps}{2} \sumii\|u_i\|_{\LQM{q}}  + \eps \sumjj \|v_j\|_{\LQM{q}} \Big) \| \theta_{\M}\|_{\LQM{p}}.
		\end{split}
	\end{equation} 
	After applying estimates of $I_1$, $I_2$ and $I_3$ into \eqref{Lem:Dual:v:Proof4}, we obtain the estimate  
	\begin{align*} 
		\begin{split}
			\intQM v_j  \theta_{\M}     
			& \leq     
			\Big(C_{\eps,T} +   \eps\Big( \sumii\|u_i\|_{\LQ{q}}  + \sumii \|u_i\|_{\LQM{q}} + \sumjj \|v_j\|_{\LQM{q}}\Big)  \Big) \| \theta_{\M}\|_{\LQM{p}} \\
			& \quad + C_{T}  \|v_k\|_{\LQM{q}}    \| \theta_{\M}\|_{\LQM{p}}.
		\end{split} 
	\end{align*} 
	By duality, we arrive at the desired result \eqref{Lem:Dual:v:State}. 
\end{proof}

\subsection{The proof of Theorem \ref{Theorem:1}: Global existence}

Combining Lemma \ref{Lem:L^p:energy} and 	Lemma \ref{Lem:Dual:v}, we obtain Lemma \ref{Lem:relative:v}.
\begin{lemma}\label{Lem:relative:v}
	Assume \eqref{A1}, \eqref{A2}, \eqref{A3} and \eqref{A4} with \eqref{Condition:Growth}.
	For any $q>1$, any $\eta>0$ and any $j\in \{1,\cdots, m_2\}$, there exists a constant $C_{T,\eta}$ depending on $T$ and $\eta$ such that  
	\begin{equation}\label{Lem:relative:v:State1}
		\|v_j\|_{\LQM{q}} 
		\leq  C_{T,\eta}   +  \eta \sumjj\|v_j\|_{\LQM{q}} + C \sum_{k=1}^{j-1} \|v_k\|_{\LQM{q}}.  
	\end{equation}
	In particular, when $j = 1$, the first sum on the right-hand side is neglected.
\end{lemma}		
\begin{proof}
	From Lemma \ref{Lem:Dual:v}, we obtain
	\begin{align*}
			\|v_j\|_{\LQM{q}}       
			\leq C + \eps\Big( \sumii\|u_i\|_{\LQ{q}}  + \sumii \|u_i\|_{\LQM{q}} + \sumjj \|v_j\|_{\LQM{q}}\Big)    + C \sum_{k=1}^{j-1} \|v_k\|_{\LQM{q}}. 
	\end{align*} 
	Combining Lemma \ref{Lem:L^p:energy}, we have
	\begin{align*} 
		\|v_j\|_{\LQM{q}} 
		& \leq  C_{T,\eps} +   \eps( \eps  +1)\sumjj\|v_j\|_{\LQM{q}} + C \sum_{k=1}^{j-1} \|v_k\|_{\LQM{q}}. 
	\end{align*} 
	Letting $\eta= \eps( \eps  +1)$, we arrive at the estimate \eqref{Lem:relative:v:State1}.
\end{proof}

\begin{lemma}\label{Lem:Dual:uv:v}
	Assume \eqref{A1}, \eqref{A2}, \eqref{A3} and \eqref{A4} with \eqref{Condition:Growth}.
	It holds that
	\begin{align}\label{Lem:Dual:uv:v:State} 
		\sumii \left(\|u_i\|_{\LQ{q}}+\|u_i\|_{\LQM{q}} \right) + \sumjj 	\|v_j\|_{\LQM{q}}  \leq C_T.
	\end{align}
\end{lemma}			
\begin{proof} 
	From Lemma \ref{Lem:relative:v}, we  obtain for all $j=1,\ldots, m_2$
	\begin{align*}
		\|v_j\|_{\LQM{q}} 
		\leq  C_{T,\eta}   +  \eta \sumjj\|v_j\|_{\LQM{q}} + C \sum_{k=1}^{j-1}\|v_k\|_{\LQM{q}}. 
	\end{align*}
	By using Lemma \ref{Lem:key}, we arrive at the estimate   
	\begin{align}\label{Lem:Dual:uv:v:Proof1}
		\sumjj 	\|v_j\|_{\LQM{q}}  \leq C_T.
	\end{align}
	Using this in \eqref{Lem:L^p:energy:State2} from Lemma \ref{Lem:L^p:energy}, we have
	\begin{align}\label{Lem:Dual:uv:v:Proof3} 
		\sumii \left(\|u_i\|_{\LQ{q}}+\|u_i\|_{\LQM{q}} \right) \leq C_T.
	\end{align}
	Estimate \eqref{Lem:Dual:uv:v:State} is a direct result of \eqref{Lem:Dual:uv:v:Proof1} and \eqref{Lem:Dual:uv:v:Proof3}.
\end{proof}

\begin{proof}[The proof of Theorem \ref{Theorem:1}: Global existence]
	From Lemma \ref{Lem:Dual:uv:v}, for any $p>1$, we derive the estimate
	\begin{align*}  
		\sumii \left(\|u_i\|_{\LQ{p}}+\|u_i\|_{\LQM{p}} \right) + \sumjj 	\|v_j\|_{\LQM{p}}  \leq C_T.
	\end{align*}
	According to condition \eqref{A5},  
	it follows that 
	\begin{align*}
		\begin{cases}
			\pa_t u_{i} - d_i\Delta u_i = F_i(u) \leq K_2\left(|u|^{r}+ 1\right)\in \LQ{p},\vspace*{0.1cm}\\
			d_{i} \pa_{\nu} u_{i}  = G_i(u,v) \leq K_2\left(|u|^{r}+|v|^{r}+1\right) \in \LQM{p}, 
		\end{cases}
	\end{align*}
	for any $p> 1$. Consider the system
	\begin{align*}
		\begin{cases}
			\pa_t \tilde{u}_{i} - d_i\Delta \tilde{u}_i = F_i(u) = K_2\left(|u|^{r}+ 1\right)\in \LQ{p},\vspace*{0.1cm}\\
			d_{i} \pa_{\nu} \tilde{u}_{i}  = G_i(u,v) = K_2\left(|u|^{r}+|v|^{r}+1\right) \in \LQM{p}, 
		\end{cases}
	\end{align*} 
	for arbitrary $p> n+2$.
	According to the regularizing effect of linear parabolic equations with inhomogeneous boundary conditions (\cite[Theorem 3.2]{nittka2014inhomogeneous}), it follows that $\|\tilde{u}_i\|_{\LQ{\infty}}$ is bounded. 
	Applying comparison principle, we get $\|u_i\|_{\LQ{\infty}}$ is bounded.
	Similarly, $\|v_j\|_{\LS{\infty}}$ is bounded. 
	This shows that the solution $(u,v)$ is bounded in $\LQ{\infty}^{m_1}\times\LQM{\infty}^{m_2} $, ensuring the desired global existence.
\end{proof}

\section{Uniform-in-time bounds}\label{Sec:uniform-in-time}
In this section, we will complete the proof of Theorem \ref{Theorem:1} by showing the  the uniform-in-time bound of the solution. For this purpose, we study the system \eqref{System} on each cylinder $Q_{\tau,\tau+2} = \Omega\times(\tau,\tau+2)$, $\tau\in \mathbb N$. For the rest of this section, all constants are independent of $\tau$ unless otherwise stated.	
We define a smooth cutoff function $\varphi\in C^\infty(\R;[0,1])$ such that 
\begin{align*}
	\varphi(s)=
	\begin{cases}
		0,\quad & s\leq 0\\
		1,\quad & s\geq 1.
	\end{cases}
\end{align*}
Moreover, we assume that $|\varphi'(s)|\leq 2$ for all $s\in\R$. For any $\tau\in\N$, the shifted cutoff function is defined as $\varphi_\tau(\cdot) = \varphi(\cdot - \tau)$. 

By multiplying the system \eqref{System} by $\vat$ we have the truncated system:
\begin{equation}\label{shifted_sys}
	\begin{cases}
		\pa_t(\vat u_i) = d_i\Delta(\vat u_i) + \vat' u_i + \vat F_i(u), &(x,t)\in Q_{\tau,\tau+2}, \vspace*{0.1cm} \\
		d_i\pa_{\nu}(\vat u_i) = \vat G_i(u,v), &(x,t)\in \M_{\tau,\tau+2}, \vspace*{0.1cm}\\
		\pa_t(\vat v_j) = \delta_j\Delta_{\M}(\vat v_j) + \vat' v_j + \vat H_j(u,v), &(x,t)\in \M_{\tau,\tau+2} 
	\end{cases}
\end{equation}
with {\it zero initial data}
\begin{equation}\label{zero_initial}
	\begin{cases}
		(\vat u_i)(x,\tau) = 0, & x\in\Omega, \\
		(\vat v_j)(x,\tau) = 0, & x\in \M,
	\end{cases}
\end{equation}
where $i=1,\cdots , m_1$ and $j=1,\cdots , m_2$.

The proof of the following lemma follows a from a similar approach to Lemma \ref{Lem:L_xL^1_t}.
\begin{lemma}\label{Lem:L_xL^1_tau+}
	Assume \eqref{A1}, \eqref{A2} and \eqref{A3} with $L<0$ or $L=K=0$. It holds that 
	\begin{equation}\label{uniform-L1}
		\sup_{t\ge 0}\sumi \|u_i(t)\|_{\LO{1}} + \sup_{t\ge 0}\sumj \|v_j(t)\|_{\LM{1}} + \sup_{\tau\in \mathbb N}\sumi \|u_i\|_{L^{1}(M_{\tau,\tau+1})} \le C,
	\end{equation}
	and
	\begin{align*}
		\Big\|\int^{\tau+2}_{\tau} \vat u_i  \,ds\Big\|_{\LO{\infty}} + \Big\|\int^{\tau+2}_{\tau} \vat u_i \,ds\Big\|_{\LM{\infty}}  + \Big\|\int^{\tau+2}_{\tau} \vat v_j \,ds\Big\|_{\LM{\infty}} \leq C 
	\end{align*}
	for any $i=1,\cdots,m_1$ and $j=1,\cdots,m_2$.
\end{lemma}

\begin{proof}
	The estimate \eqref{uniform-L1} was proved in \cite[Lemma 4.4]{morgan2023global}.
	By summing the equation of \eqref{System}, we have  
	\begin{align}\label{Lem:L_xL^1_tau+:Sys1}
		\begin{cases}
			\pa_t \(\sumii  (\vat u_i)\) - \Delta  \left( \sumii d_i \vat u_i\right)  = \vat' \sumii u_i  + \vat\sumii  F_i(u), \vspace*{0.2cm}\\
			\nabla (\sumii d_i \vat u_i)\cdot \nu = \vat\sumii  G_i(u,v), \vspace*{0.2cm}\\
			\pa_t \(\sumjj (\vat v_j)\) - \Delta_{\M}  \left( \sumjj \delta_j \vat v_j\right)  =  \vat'\sumjj v_j + \vat \sumjj  H_j(u, v), \vspace*{0.2cm}\\
			\sumii  (\vat u_i)(x,\tau) = 0, \quad \sumjj  (\vat v_j)(x,\tau) = 0.
		\end{cases}  
	\end{align}
	We define for $t\in (\tau,\tau+2)$
	\begin{align*}
		\underline{W}(x,t): = \int^{t}_{\tau} \Big( \sumii d_i \vat u_i(x,s)  \Big)  \,ds 
		\quad \text{and}\quad 
		\underline{Z}(x,t): = \int^{t}_{\tau} \Big( \sumjj \delta_j \vat v_j(x,s)\Big) \,ds. 
	\end{align*}
	Integrating equations \eqref{Lem:L_xL^1_tau+:Sys1} over $(\tau,t)$ and using the condition \eqref{A3}, this gives 
	\begin{align}\label{System:aux:tau+}
		\begin{cases}
			\underline{A} \pa_t \underline{W} - \Delta \underline{W}  \leq \overline{\Lambda} \underline{W} + L(t-\tau), \vspace*{0.15cm}\\
			\nabla \underline{W} \cdot \nu \leq  \overline{\Lambda} \underline{W} +  \overline{\lambda} \underline{Z} + L(t-\tau), \vspace*{0.15cm}\\
			\underline{B} \pa_t \underline{Z} - \Delta_{\M} \underline{Z} \leq \overline{\Lambda} \underline{W} + \overline{\lambda} \underline{Z}+ L(t-\tau) ,\vspace*{0.15cm}\\
			\underline{W} (x,\tau) =0, \quad \underline{Z} (x,\tau) =0,
		\end{cases}
	\end{align} 
	where $\overline{\Lambda} = \frac{|L| + 2}{\min\{d_i\}}$ and $\overline{\lambda} = \frac{|L| + 2}{\min\{\delta_j\}}$ and $\underline{A}$, $\underline{B}$ satisfy the following properties
	\begin{align*}
		0< \frac{1}{\max\{d_1,\cdots,d_{m_1}\}}:=\underline{a}   \leq \underline{A} = \frac{\sumii \vat u_i}{\sumii d_i \vat u_i} \leq \overline{a}=:\frac{1}{\min\{d_1,\cdots,d_{m_1}\}}
	\end{align*} 
	and  
	\begin{align*}
		0< \frac{1}{\max\{\delta_1,\cdots,\delta_{m_2}\}}:=\underline{b}\leq\underline{B} = \frac{  \sumjj  \vat v_j}{\sumjj \delta_j \vat v_j}\leq \overline{b}=:\frac{1}{\min\{\delta_1,\cdots,\delta_{m_2}\}}.
	\end{align*}
	We can now repeat the Moser iteration in the proof of Lemma \ref{Lem:L_xL^1_t} to get
	\begin{align*}
		\|\underline{W} \|_{\LQtaut{\infty}} + \|\underline{W}\|_{\LStaut{\infty}} + \|\underline{Z} \|_{\LStaut{\infty}} \leq C
	\end{align*} 
	where it is reminded that $C$ is independent of $\tau$. In particular, for $t\in [\tau,\tau+2] $, we have
	\begin{align*}
		\Big\|\int^{t}_{\tau} & ( \sumii d_i \vat  u_i ) \,ds \Big\|_{\LQtaut{\infty}} +  \Big\|\int^{t}_{\tau}  ( \sumii d_i \vat u_i )\,ds \Big\|_{\LStaut{\infty}}  \\
		&+ \Big\|\int^{t}_{\tau} ( \sumjj \delta_j \vat v_j )\,ds \Big\|_{\LStaut{\infty}} \leq C.
	\end{align*}
	This finish the prove of Lemma \ref{Lem:L_xL^1_tau+}.  
\end{proof}

\medskip
Using the bounds in Lemma \ref{Lem:L_xL^1_tau+}, we show the following intermediate estimates.
\begin{lemma} \label{Lem:uni:Lp} 
	Assume \eqref{A1}, \eqref{A2}, \eqref{A4}, and \eqref{A3} with $L<0$ or $L=K=0$. 
	For any $\tau\in \mathbb N$, $2\leq q$, any $k\in \{1,\cdots , m_2\}$, and any $\eps>0$, there exists a constant $C_\eps>0$ such that
	\begin{equation}\label{Lem:uni:Lp:State1}
		\begin{aligned}
			\|\vat v_k\|_{\LStaut{q}} \leq & C_\eps + C_\eps\sum_{j=1}^{k-1}\|\vat v_j\|_{\LStaut{q}} \\
			& + \eps \Big(\sumii\|u_i\|_{\LQtaut{q}} +  \sumii\|u_i\|_{\LStaut{q}}  +  \sumjj\| v_j\|_{\LStaut{q}} \Big).
		\end{aligned}
	\end{equation}
	Consequently, for any $\eps>0$, there exists $C_\eps>0$ such that 
	\begin{equation}\label{Lem:uni:Lp:State2}
		\|\vat v_k\|_{\LStaut{q}} \leq C_\eps + \eps \sumii \Big( \|u_i\|_{\LQtaut{q}} +    \|u_i\|_{\LStaut{q}} \Big) + \eps\sumjj\|v_j\|_{\LStaut{q}} 
	\end{equation}
	for all $k=1,\cdots , m_2$.
\end{lemma}
\begin{proof}
	Similarly to Lemma \ref{Lem:Dual:v}, we just need to consider the case $\rO \leq \mu_{\M} <2$.	
	Let $0\leq  \theta_{\M} \in \LStaut{p}$ with $\| \theta_{\M}\|_{\LStaut{p}} = 1$, and $\Psi_j$ be the solution to \eqref{System_Dual} with $T = \tau+2$. By direct calculations, we write
	\begin{equation*} 
		\begin{aligned}
			\intMtautwo (\vat v_j) \theta_{\M} &= \intMtautwo (\vat v_j)(-\pa_t \Psi_j - \delta_j \Delta_{\M} \Psi_j)\\
			&= \intMtautwo \Psi_j \vat' v_j +  \intMtautwo \Psi_j\vat H_j(u,v).
		\end{aligned}
	\end{equation*}
	Using the assumption  \eqref{A4} we have for any $k = 1,\cdots , m_2$ 
	\begin{equation*}
		\sumii a_{(j+m_1)i}G_i(u,v) + \sum_{k=1}^j a_{(j+m_1)(k+m_1)}H_k(u,v) \leq K_1 \left(|u|^{\mM} + |v|^{\mM} + 1 \right),
	\end{equation*}
	which implies, 
	for all $j=1,\cdots ,  m_2$,
	\begin{equation}\label{Lem:uni:Lp:Proof2}
		\begin{aligned}
			H_j(u,v)
			& \leq -\sumii \frac{a_{(j+m_1)i}}{a_{(j+m_1)(j+m_1)}}G_i(u,v) - \sum_{k=1}^{j-1} \frac{a_{(j+m_1)(k+m_1)}}{a_{(j+m_1)(j+m_1)}} H_k(u,v)\\
			&\quad + \frac{K_1}{a_{(j+m_1)(j+m_1)}}\left(|u|^{\mM} + |v|^{\mM} + 1 \right).
		\end{aligned}
	\end{equation}
	From above estimate, we get 
	\begin{align}\label{Lem:uni:Lp:Proof3}
		\begin{aligned}
			\intMtautwo (\vat v_j) \theta_{\M}      
			& \leq  \intMtautwo \Psi_j \vat' v_j  - \sumii \frac{a_{(j+m_1)i}}{a_{(j+m_1)(j+m_1)}} \intMtautwo \Psi_j  \vat G_i(u,v) \\
			&\quad - \sum_{k=1}^{j-1} \frac{a_{(j+m_1)(k+m_1)}}{a_{(j+m_1)(j+m_1)}} \intMtautwo \Psi_j \vat   H_k(u,v) \\
			&\quad +\frac{K_1}{a_{(j+m_1)(j+m_1)}} \intMtautwo \Psi_j  \vat \left(|u|^{\mM} + |v|^{\mM} + 1 \right) \\
			&= :\sum^4_{k=1} J_k.
		\end{aligned}
	\end{align} 
	Now, we will estimate each term on the right-hand side of \eqref{Lem:uni:Lp:Proof3} separately.

	\medskip
	\textbf{Estimate $J_1$}:  By applying H\"{o}lder's inequality and   Lemma  \ref{Lem:DualProblem}, we obtain 
	\begin{align*}
		J_1 :&= \intMtautwo \Psi_j \vat' v_j  \leq C\|\Psi_j\|_{\LStaut{p^*}}\|v_j\|_{\LStaut{\frac{p^*}{p^*-1}}},  
	\end{align*}
	where $p<p^*\leq \frac{(n+2)p}{n+2-2p}$.   
	Since $1<\frac{p^*}{p^*-1}<q=\frac{p}{p-1}$, we have $1<\frac{p^*}{p^*-1}<q$. By interpolation inequality, we  obtain
	\begin{align*}
		\|v_j\|_{\LStaut{\frac{p^*}{p^*-1}}} &\leq \|v_j\|_{\LStaut{1}}^{\alpha}\|v_j\|_{\LStaut{q}}^{1-\alpha}  \leq C\|v_j\|_{\LStaut{q}}^{1-\alpha}  
	\end{align*}
	thanks to \eqref{uniform-L1}, 	with $\alpha\in (0,1)$ satisfying $\frac{p^*-1}{p^*}= \alpha+ \frac{1-\alpha}{q}$. Therefore,
	\begin{equation}\label{Lem:uni:Lp:Proof4}
		|J_1| \leq C\|\Psi_j\|_{\LStaut{p^*}} \|v_j\|_{\LStaut{q}}^{1-\alpha} \leq C_\eps + \frac{\eps}{2} \|v_j\|_{\LStaut{q}},
	\end{equation}
	where we used $0\leq  \theta_{\M} \in \LStaut{p}$ and $\| \theta_{\M}\|_{\LStaut{p}} = 1$.
	
	\medskip
	\textbf{Estimate $J_2$}: 
	\begin{align*}
		J_2:&= - \sumii \frac{a_{(j+m_1)i}}{a_{(j+m_1)(j+m_1)}} \intMtautwo \Psi_j  \vat G_i(u,v) \\
		&    
		=  - \sumii \frac{a_{(j+m_1)i}}{a_{(j+m_1)(j+m_1)}} \intMtautwo \Psi_j   d_i\nabla ( \vat u_i) \cdot \nu \\
		& =  - \sumii \frac{a_{(j+m_1)i}}{a_{(j+m_1)(j+m_1)}} \left(\intQtautwo d_i \Delta (\vat u_i )\Psi_j  + \intMtautwo  d_i \vat u_i\nabla\Psi_j    \cdot \nu \right)\\
		& \quad+  \sumii \frac{a_{(j+m_1)i}}{a_{(j+m_1)(j+m_1)}} \intQtautwo d_i \vat u_i \Delta \Psi_j  \\ 
		& =     \sumii \frac{a_{(j+m_1)i}}{a_{(j+m_1)(j+m_1)}} \intQtautwo \vat' u_i \Psi_j +   \sumii \frac{a_{(j+m_1)i}}{a_{(j+m_1)(j+m_1)}} \intQtautwo \vat F_i(u) \Psi_j  \\
		&\quad - \sumii \frac{a_{(j+m_1)i}}{a_{(j+m_1)(j+m_1)}}  \intMtautwo  d_i\vat u_i\nabla\Psi_j \cdot \nu \\
		& \quad+  \sumii \frac{a_{(j+m_1)i}}{a_{(j+m_1)(j+m_1)}}  \intQtautwo  \vat  u_i (\pa_t \Psi_j  + d_i\Delta \Psi_j  ) =: \sum^4_{k=1} J_{2k}.
	\end{align*}
	
	In order to estimate $J_2$, we need to estimate $J_{2k}, ~k=1,\cdots,4$.
	
	\underline{Estimate of $J_{21}$: } Similar to the estimate of $J_1$, we  obtain
	\begin{align*}
		J_{21} : &=   \sumii \frac{a_{(j+m_1)i}}{a_{(j+m_1)(j+m_1)}} \intQtautwo \vat' u_i \Psi_j \leq C_\eps + \frac{\eps}{3} \|u_i\|_{\LQtaut{q}}.   
	\end{align*}
	
	\underline{Estimate of $J_{22}$: }   
	From condition \eqref{A4}, we have
	\begin{align*}
		\sumii a_{(j+m_1)i} F_i(u) \leq K_1 \left(|u|^{\rO} + 1\right).
	\end{align*}
	Therefore, by applying H\"{o}lder's inequality and  Lemma  \ref{Lem:DualProblem}, we arrive at
	\begin{equation*} 
		\begin{split}
			J_{22}:&= \sumii \frac{a_{(j+m_1)i}}{a_{(j+m_1)(j+m_1)}} \intQtautwo \vat F_i(u) \Psi_j  \leq C \intQtautwo \Psi_j \left(|u|^{\rO} + 1\right)\\ 
			&\leq C\intQtautwo  \Psi_j   + C  \sumi \intQtautwo  u_i^{\rO} \Psi_j    \\ 
			&\leq   C\| \theta_{\M}\|_{\LStaut{p}}   +   C\sumi\intO |\Psi_j |_{L^{\infty}(\tau,\tau+2)} \int^{\tau+2}_{\tau} u_i^{\rO} \,dsdx \\
			&\leq  C    + C \sumi\||\Psi_j |_{L^{\infty}(\tau,\tau+2)}\|_{\LO{p}} \(\intO  \left(\int^{\tau+2}_{\tau} u_i^{\rO}\,ds\)^q dx\right)^{\frac{1}{q}},
		\end{split}
	\end{equation*} 
	where $q=\frac{p}{p-1}$. 
	For $\rO= 1$, by applying Lemmas \ref{Lem:L_xL^1_tau+} and \ref{Lem:DualProblem}, we get
	\begin{align*}
		\begin{split}
			J_{22} \leq  C.  
		\end{split}
	\end{align*}
	Next, we consider $1<\rO< 1+ \frac{2}{n}$,  take  $k= p-\mu_{\M}(p-1) $, it noted that $\frac{\rO-k}{1-k} \leq \frac{\mu_{\M}-k}{1-k}  =q=\frac{p}{p-1}$ and $ (1-k) q =  (1-k) \frac{p}{p-1}  <1$. 
	\begin{equation*} 
		\begin{split}
			&  \(\intO  \(\int^{\tau+2}_{\tau} u_i^{\rO}\,ds\)^q dx\)^{\frac{1}{q}}
			\leq  \(\intO  \(\int^{\tau+2}_{\tau} u_i^{\rO-k} u_i^k\,ds\)^q dx\)^{\frac{1}{q}} \\ 
			&\leq   \Big\|\int^{\tau+2}_{\tau} u_i \,ds\Big\|_{\LO{\infty}}^{k}  \(\intO  \int^{\tau+2}_{\tau} u_i^{\frac{\rO-k}{1-k}}  \,ds \,dx\)^{1-k} |\Omega|^{\frac{1-(1-k)q}{q}} \\
			&\leq C\|u_i\|_{\LQtaut{q}}^{1-k},
		\end{split}
	\end{equation*} 
	where we used the estimates of Lemma \ref{Lem:L_xL^1_tau+}.
	Combining the above estimates and using Young's inequality, we  obtain 
	\begin{align*}
		J_{22} 
		\leq  C + \frac{\eps}{3} \sumii\|u_i\|_{\LQtaut{q}}.  
	\end{align*}

	\underline{Estimate of $J_{23}$: } With the help of H\"{o}lder's inequality and the estimates of Lemma  \ref{Lem:DualProblem}, we arrive at
	\begin{align*}
		J_{23}: &= - \sumii\frac{a_{(j+m_1)i}}{a_{(j+m_1)(j+m_1)}}\intMtautwo d_i\vat u_i\nabla\Psi_j \cdot\nu \\
		& \leq C\sumii\|u_i\|_{\LStaut{\frac{p+\xi}{p+\xi-1}}}\|\partial_{\nu}\Psi_j \|_{\LStaut{p+\xi}}\\
		&\leq C\sumii\|u_i\|_{\LStaut{\frac{p+\xi}{p+\xi-1}}}\| \theta_{\M}\|_{\LStaut{p}}\\
		& \leq \frac{\eps}{2} \sumii\|u_i\|_{\LStaut{q}} + C,
	\end{align*} 
	where $\xi$ come from Lemma \ref{Lem:DualProblem}.
	
	\underline{Estimate of $J_{24}$: } Using H\"{o}lder's inequality together with the estimates provided in Lemma \ref{Lem:DualProblem}, we obtain
	\begin{align*}
		J_{24} : &= \sumii \frac{a_{(j+m_1)i}}{a_{(j+m_1)(j+m_1)}} \intQtautwo \vat u_i (\pa_t \Psi_j  + d_i\Delta \Psi_j  )   \leq  C + \frac{\eps}{3} \sumii\|u_i\|_{\LQtaut{q}}. 
	\end{align*} 
	By combining the estimates of $J_{21}$--$J_{24}$ with $J_2$, we  obtain
	\begin{align*}
		& J_{2} \leq \frac{\eps}{2} \sumii\|u_i\|_{\LStaut{q}} + \eps\sumii \|u_i\|_{\LQtaut{q}} + C. 
	\end{align*} 
	
	\medskip
	\textbf{Estimate $J_3$}: 
	\begin{align*}
		J_3:&=  - \sum_{k=1}^{j-1} \frac{a_{(j+m_1)(k+m_1)}}{a_{(j+m_1)(j+m_1)}} \intMtautwo \Psi_j \vat   H_k(u,v) \\  
		& = - \sum_{k=1}^{j-1} \frac{a_{(j+m_1)(k+m_1)}}{a_{(j+m_1)(j+m_1)}} \intMtautwo \Psi_j   \left(\pa_t(\vat v_k) - \delta_k\Delta_{\M}(\vat v_k) - \vat' v_k \right)\\
		& =  \sum_{k=1}^{j-1} \frac{a_{(j+m_1)(k+m_1)}}{a_{(j+m_1)(j+m_1)}} \intMtautwo \vat v_k  \left(\pa_t \Psi_j  + \delta_k \Delta \Psi_j \right)  \\
		&\quad +   \sum_{k=1}^{j-1} \frac{a_{(j+m_1)(k+m_1)}}{a_{(j+m_1)(j+m_1)}} \intMtautwo \Psi_j   \vat' v_k\\
		& \leq  C + C \sum_{k=1}^{j-1} \frac{a_{(j+m_1)(k+m_1)}}{a_{(j+m_1)(j+m_1)}} \|\vat v_k\|_{\LStaut{q}}.  
	\end{align*}
	
	\medskip
	\textbf{Estimate $J_4$}:  Similar to the prove of estimate $J_{22}$, we have
	\begin{align*}
		J_4:&=  \frac{K_1}{a_{(j+m_1)(j+m_1)}} \intMtautwo \Psi_j  \vat \left(\sumii u_i^{\mM} + \sumjj v_j^{\mM} + 1 \right) \\
		& \leq C+ \frac{\eps}{2} \sumii\|u_i\|_{\LStaut{q}}  + \frac{\eps}{2} \sumjj\|v_j\|_{\LStaut{q}}. 
	\end{align*} 
	
	\medskip
	Applying all the estimates of $J_1$, $J_2$, $J_3$ and $J_4$ in \eqref{Lem:uni:Lp:Proof4},  we obtain
	\begin{equation*}
		\begin{aligned}
			\intMtautwo (\vat v_j) \theta_{\M}&\leq C_\eps + C \sum_{j=1}^{k-1}\|\vat v_j\|_{\LStaut{q}}\\ 
			&+ \eps \Big(\sumii\|u_i\|_{\LQtaut{q}} +  \sumii\|u_i\|_{\LStaut{q}}  +  \sumjj\| v_j\|_{\LStaut{q}} \Big). 
		\end{aligned}
	\end{equation*} 
	From this we get the estimate \eqref{Lem:uni:Lp:State1} due to duality.  Finally \eqref{Lem:uni:Lp:State2} follows from \eqref{Lem:uni:Lp:State1} by induction.
\end{proof}

\subsection{The proof of Theorem \ref{Theorem:1}: Uniform-in-time bounds}

Lemma \ref{Lem:LpEnergy:tau} can be obtained by using $L^p$-energy function, The detailed proof is provided in \cite{morgan2023global}.
\begin{lemma}[\cite{morgan2023global}]\label{Lem:LpEnergy:tau} 
	Assume \eqref{A1}, \eqref{A2}, \eqref{A3} and \eqref{A4}  with \eqref{Condition:Growth}. If $L<0$ or $L=K=0$ in \eqref{A3}, then for any positive integer $p\ge 2$, and any $\eps>0$, there exists $K_{p,\eps}>0$ such that
	\begin{equation}\label{Lem:LpEnergy:tau:State1}
		\begin{aligned}
			&\sumii   \left(\intQtautwo (\vat u_i)^{p-1+\rO} + \intMtautwo (\vat u_i)^{p-1+\rM}\right)\\
			&\leq K_{p,\eps} + \eps \left(\sumii  \left(\intQtautwo u_i^{p-1+\rM} + \intMtautwo u_i^{p-1+\rM}\right)+ \sumjj  \intMtautwo v_j^{p-1+\rM} \right).
		\end{aligned}
	\end{equation} 
	As a consequence, for any $1<p<\infty$ and any $\eps>0$, there exists $K_{p,\eps}>0$ such that
	\begin{equation}\label{Lem:LpEnergy:tau:State2}
		\begin{aligned}
			&\sumii  \left(\|\varphi_\tau u_i\|_{\LQtaut{p}} + \|\varphi_\tau u_i\|_{\LStaut{p}}\right)\\
			&\leq K_{p,\eps} + \eps \left(\sumii  \left(\|u_i\|_{\LQtaut{p}} + \|u_i\|_{\LStaut{p}}\right) + \sumjj \|v_j\|_{\LStaut{p}} \right).
		\end{aligned}
	\end{equation}
\end{lemma}

We need the following elementary result whose proof is straightforward.
\begin{lemma}\label{Lem:Uni_elementary}
	Let $\{y_N\}_{N\ge 0}$ be a nonnegative sequence and $\mathscr N = \{N\in \mathbb N: y_{N-1}\leq y_N \}$. If there exists $K>0$ (independent of $N$) such that
	\begin{equation*}
		y_N \leq K \quad \text{ for all } \, N\in \mathscr N,
	\end{equation*}
	then
	\begin{equation*}
		y_N \leq \max\{y_0, K\} \quad \text{ for all } \, N\in \mathbb N.
	\end{equation*}
\end{lemma} 

\begin{lemma}\label{Lem:Lp:uv}
	Assume \eqref{A1}, \eqref{A2}, \eqref{A3} and \eqref{A4}  with \eqref{Condition:Growth}.
	It holds that
	for any $2\leq p$, there exists a constant $C_p>0$ such that
	\begin{equation}\label{Lem:Lp:uv:State}
		\sumii   \left(\|u_i\|_{\LQtau{p}} + \|u_i\|_{\LStau{p}}\right) + \sumjj \|v_j\|_{\LStau{p}} \leq C_p \quad \text{ for all } \, \tau\in \mathbb N.
	\end{equation}
\end{lemma}

\begin{proof} 
	From \eqref{Lem:LpEnergy:tau:State2} in Lemma \ref{Lem:LpEnergy:tau} and \eqref{Lem:uni:Lp:State2} in Lemma \ref{Lem:uni:Lp}, we get for any $\eps>0$, there exists a constant $C_\eps>0$ such that
	\begin{align}\label{Lem:Lp:uv:Proof1}
		\begin{split}
			&\sumii\left(\|\vat u_i\|_{\LQtaut{p}} + \|\vat u_i\|_{\LStaut{p}}\right) + \sumjj \|\vat v_j\|_{\LStaut{p}} \\
			&\leq C_\eps + \eps \left(\sumii\left(\|u_i\|_{\LQtaut{p}} + \|u_i\|_{\LStaut{p}}\right)+ \sumjj\|v_j\|_{\LStaut{p}} \right).
		\end{split}
	\end{align}
	It is noted  that $\vat\geq 0$ and $\vat|_{[\tau,\tau+1]}\equiv 1$, it follows from \eqref{Lem:Lp:uv:Proof1} that
	\begin{equation}\label{Lem:Lp:uv:Proof2}
		\begin{aligned}
			&\sumii  \left(\|u_i\|_{\LQtau{p}} + \|u_i\|_{\LStau{p}} \right) +  \sumjj \|v_j\|_{\LStau{p}}\\
			&\leq C_\eps + C \eps \left( \sumii\left(\|u_i\|_{\LQtaut{p}} + \|u_i\|_{\LStaut{p}} \right) + \sumjj\|v_j\|_{\LStaut{p}} \right).
		\end{aligned}
	\end{equation}
	For $\tau \in \mathbb N$, we define 
	\begin{equation*}
		y_{\tau}:= \sumii\left(\|u_i\|_{\LQtau{p}} + \|u_i\|_{\LStau{p}}\right) + \sumjj\|v_j\|_{\LStau{p}}.
	\end{equation*}
	Inequality \eqref{Lem:Lp:uv:Proof2} implies
	\begin{equation}\label{Lem:Lp:uv:Proof3}
		y_\tau \leq C +  C \eps(y_\tau + y_{\tau+1}).
	\end{equation}
	Define $\mathscr N  = \{\tau \in \mathbb N: y_{\tau}\leq y_{\tau+1}\}$. From Lemma \ref{Lem:key},  for any $\tau\in \mathscr N$, by choosing $\eps$ sufficiently small, we obtain from \eqref{Lem:Lp:uv:Proof3}
	\begin{equation*}
		y_\tau \leq C,
	\end{equation*}
	where $C$ is independent of $\tau$. From Lemma \ref{Lem:Uni_elementary}, we have
	\begin{equation*}
		y_\tau \leq C \quad \text{ for all } \, \tau\in \mathbb N,
	\end{equation*}
	which proves the desired estimate  \eqref{Lem:Lp:uv:State}.
	
\end{proof}

We are now ready to show the uniform-in-time bound in Theorem \ref{Theorem:1}. 
\begin{proof}[The proof of Theorem \ref{Theorem:1}: uniform-in-time bound]
	Now, by using \eqref{Lem:Lp:uv:State} in the truncated system \eqref{shifted_sys} and choosing $p$ to be large enough, using same way with the proof of global existence, we can conclude that there exists $C_\infty>0$ independent of $\tau\in \mathbb N$ such that
	\begin{equation*}
		\sumii \|u_i\|_{\LQtau{\infty}} + \sumjj \|v_j\|_{\LStau{\infty}} \leq C_\infty \quad \text{ for all } \, \tau\in \mathbb N,
	\end{equation*}
	which finishes the proof of Theorem \ref{Theorem:1}.
\end{proof}

\medskip
\subsection*{Acknowledgement.} J. Yang is supported by NSFC Grants No. 12271227 and China Scholarship Council (Contract No. 202206180025). B.Q. Tang receives funding from the FWF project
``Quasi-steady-state approximation for PDE'', number I-5213.

\newcommand{\etalchar}[1]{$^{#1}$}

\end{document}